\documentclass[11pt,twoside,a4paper]{amsart}

\usepackage{amssymb}
\usepackage{vmargin}
\usepackage{amscd}
\usepackage{stmaryrd}
\usepackage{mathrsfs}
\usepackage[all]{xy}

\usepackage{xr}
\usepackage{hyperref}
\hypersetup{colorlinks,
linkcolor=black,
citecolor=black,
urlcolor=blue}

\setmargins{32mm}{20mm}{14.6cm}{22cm}{1cm}{1cm}{1cm}{1cm}

\newcommand\id{\mathrm{id}}

\newcommand\ten{\otimes}

\newcommand\eps{\epsilon}

\newcommand\OO{\mathrm{O}}

\renewcommand\H{\mathrm{H}}
\newcommand\z{\mathrm{Z}}

\newcommand\Z{\mathbb{Z}}
\newcommand\Q{\mathbb{Q}}

\newcommand\R{\mathbb{R}}
\newcommand\Cx{\mathbb{C}}

\newcommand\bA{\mathbb{A}}

\newcommand\bD{\mathbb{D}}

\newcommand\bG{\mathbb{G}}
\newcommand\bH{\mathbb{H}}

\newcommand\bM{\mathbb{M}}

\newcommand\C{\mathcal{C}}

\newcommand\cP{\mathcal{P}}

\renewcommand\O{\mathscr{O}}

\newcommand\sD{\mathscr{D}}
\newcommand\sE{\mathscr{E}}
\newcommand\sF{\mathscr{F}}
\newcommand\sG{\mathscr{G}}

\newcommand\sI{\mathscr{I}}

\newcommand\sL{\mathscr{L}}

\newcommand\sO{\mathscr{O}}

\newcommand\sT{\mathscr{T}}
\newcommand\sU{\mathscr{U}}
\newcommand\sV{\mathscr{V}}
\newcommand\sW{\mathscr{W}}

\newcommand\fX{\mathfrak{X}}

\renewcommand\L{\Lambda}

\renewcommand\hom{\mathscr{H}\!\mathit{om}}

\newcommand\cDiff{\mathcal{D}\!\mathit{iff}}

\newcommand\Hom{\mathrm{Hom}}

\newcommand\HHom{\underline{\mathrm{Hom}}}
\newcommand\EEnd{\underline{\mathrm{End}}}

\newcommand\cone{\mathrm{cone}}
\newcommand\cocone{\mathrm{cocone}}

\newcommand{\brh}{\llbracket \hbar \rrbracket}

\newcommand\im{\mathrm{Im\,}}

\newcommand\ob{\mathrm{ob}}

\newcommand\Spec{\mathrm{Spec}\,}

\newcommand\Sp{\mathrm{Sp}}
\newcommand\PreSp{\mathrm{PreSp}}

\newcommand\Pol{\mathrm{Pol}}

\newcommand\Comp{\mathrm{Comp}}

\newcommand\nondeg{\mathrm{nondeg}}

\newcommand\ad{\mathrm{ad}}

\newcommand\<{\langle}
\renewcommand\>{\rangle}
\newcommand\Lim{\varprojlim}
\newcommand\LLim{\varinjlim}

\newcommand\xra{\xrightarrow}
\newcommand\xla{\xleftarrow}
\newcommand\rk{\mathrm{rk}}

\newcommand\bt{\bullet}
\newcommand\by{\times}

\newcommand\mc{\mathrm{MC}}
\newcommand\mmc{\underline{\mathrm{MC}}}

\newcommand\SO{\mathrm{SO}}
\newcommand\GL{\mathrm{GL}}

\newcommand\an{\mathrm{an}}

\newcommand\Tot{\mathrm{Tot}\,}

\newcommand\pd{\partial}

\newcommand\half{\frac{1}{2}}

\newcommand\gr{\mathrm{gr}}

\newcommand\Fil{\mathrm{Fil}}

\newcommand\red{\mathrm{red}}

\newcommand\cosk{\mathrm{cosk}}

\newcommand\dR{\mathrm{dR}}
\newcommand\DR{\mathrm{DR}}

\newcommand\op{\mathrm{opp}}

\newcommand\co{\colon\thinspace}

\newcommand\oR{\mathbf{R}}

\newcommand\oL{\mathbf{L}}

\newcommand\uleft\underleftarrow
\newcommand\uline\underline
\newcommand\uright\underrightarrow

\newtheorem{theorem}{Theorem}[section]
\newtheorem{proposition}[theorem]{Proposition}
\newtheorem{corollary}[theorem]{Corollary}

\newtheorem{lemma}[theorem]{Lemma}
\newtheorem*{theorem*}{Theorem}
\newtheorem*{proposition*}{Proposition}
\newtheorem*{corollary*}{Corollary}
\newtheorem*{lemma*}{Lemma}
\newtheorem*{conjecture*}{Conjecture}

\theoremstyle{definition}
\newtheorem{definition}[theorem]{Definition}
\newtheorem*{definition*}{Definition}

\newtheorem*{notation*}{Notation}

\theoremstyle{remark}
\newtheorem{example}[theorem]{Example}

\newtheorem{remark}[theorem]{Remark}
\newtheorem{remarks}[theorem]{Remarks}

\newtheorem*{example*}{Example}
\newtheorem*{examples*}{Examples}
\newtheorem*{remark*}{Remark}
\newtheorem*{remarks*}{Remarks}
\newtheorem*{exercise*}{Exercise}
\newtheorem*{property*}{Property}
\newtheorem*{properties*}{Properties}

\begin{document}

\begin{abstract}
We formulate a notion of $E_{-1}$ quantisation of $(-2)$-shifted Poisson structures on derived algebraic stacks, depending on a flat right connection on the structure sheaf, as solutions of a quantum master equation.  
We then parametrise    $E_{-1}$ quantisations of $(-2)$-shifted symplectic structures by constructing a map to power series in  de Rham cohomology. For derived schemes, we show that these quantisations give rise to classes  in Borel--Moore homology,  and for a large class of examples we show that the classes  are closely related to Borisov--Joyce invariants.
\end{abstract}

\title{Deformation quantisation for $(-2)$-shifted symplectic structures} 

\author{J.P.Pridham}


\maketitle

\section*{Introduction}

Shifted symplectic structures in derived algebraic geometry were introduced in \cite{PTVV}, and shifted Poisson structures, together with the correspondence between shifted symplectic structures and non-degenerate shifted Poisson structures, in \cite{poisson,CPTVV}. For $n \ge 1$, deformation quantisation of $n$-shifted Poisson structures is an immediate consequence of formality of the $E_{n+1}$ operad, as observed in \cite{CPTVV}. For $n=0$ and $n=-1$, deformation quantisation of $n$-shifted  Poisson structures is more subtle, but  was investigated and largely established in \cite{DQnonneg,DQpoisson,DQvanish}; we now look at the case  $n=-2$.

By deformation quantisation of a derived scheme or stack $X$ over $R$, we should mean some form of non-commutative formal deformation of the structure sheaf over $R\brh$, for $\hbar$ an element of homological degree $0$. In particular, this excludes the red shift quantisations proposed in \cite{CPTVV}. Meanwhile the structures enhancing fundamental classes in \cite{BBDJS,BorisovJoyce} are naturally defined over $R((\hbar))$, and should be recovered by localising deformation quantisations away from $\hbar=0$ (cf. \cite[\S \ref{DQvanish-vanishsn}]{DQvanish} for the $n=-1$ case).

For $n \ge -1$, an $n$-shifted quantisation is a (Beilinson--Drinfeld) $BD_{n+1}$-algebra. For $n \ge 0$, this is a filtered, almost commutative $E_{n+1}$-algebra deforming the $P_{n+1}$-algebra given by the Poisson structure. The category of modules over such an algebra is an $n$-tuply monoidal linear category, so $n=0$ just gives a linear category. The case $n=-1$ concerns $BD_0$-algebras, which are filtered, almost commutative Batalin--Vilkovisky (BV)-algebras. These are just  objects in a category, and a $(-2)$-shifted quantisation will just be an element of an object. Since the hierarchy of $BD_{n+1}$-algebras has petered out by $n=-2$, we make use of the observation that for $n\ge -1$, $n$-shifted quantisations are parametrised by Maurer--Cartan elements of  a natural $BD_{n+2}$-algebra (given by differential operators or Hochschild complexes) deforming the $P_{n+2}$-algebra of shifted polyvectors. 

For $n=-2$, we thus  consider the $BD_0$-algebra given by the Hodge filtration on the right de Rham complex of $\sO_X$ associated to a  flat right connection on   $\sO_X$, and formulate (Definition \ref{Qpoissdef})
 deformation quantisations of $(-2)$-shifted Poisson structures  as Maurer--Cartan elements of an associated  $BV$-algebra, i.e. as solutions of the quantum master equation. On inverting $\hbar$, this gives a Laurent series of cohomology classes in a right de Rham complex.

Our main results are Propositions \ref{quantprop} and  \ref{uniqueconn}, and their global analogues Proposition \ref{DMquantprop} and \S \ref{Artinsn}. Proposition \ref{quantprop} parametrises $E_{-1}$-quantisations in terms of first-order quantisations and power series in de Rham cohomology; in particular, it shows that the only obstruction to quantising a $(-2)$-shifted Poisson structure is first order. The strategy of proof is adapted from \cite{DQvanish}, involving  a notion of compatibility between $(-2)$-shifted Poisson structures and  de Rham power series.  Proposition \ref{uniqueconn} then shows that if there exist any flat right connections on $\sO_X$, then there is an essentially unique flat right connection admitting first-order quantisations of a given non-degenerate    $(-2)$-shifted Poisson structure. These combine to give Corollary \ref{uniqueconncor}, ensuring the existence of quantisations for $(-2)$-shifted symplectic structures whenever $\sO_X$ admits a flat right connection. 

In  \S \ref{cfBJ} we set about associating virtual fundamental classes to quantisations. Right de Rham cohomology of the dualising complex is just Borel--Moore homology, but for  $(-2)$-shifted symplectic derived schemes the dualising complex $\omega_X$ is seldom a line bundle. However, in Proposition \ref{DRrtoBMprop} and Remark \ref{DRrGorensteinRmk}, we establish a quasi-isomorphism  $\DR^r_X(\det \Omega^1_X)[\dim X] \to \DR^r_X(\omega_X)$ with the right de Rham complex of the determinant bundle. Thus  generating virtual fundamental classes in Borel--Moore homology from our quantisations just relies on a quasi-isomorphism $ \det \Omega^1_X \simeq \sO_X$ of right $\sD$-modules, which we can think of as orientation data.  

More generally, the shifted symplectic structure always defines an inner product on the line bundle $\det \Omega^1_X $, and  over $\Cx$ we can consider the local system of analytic functions $o_X \subset (\det \Omega^1_X)^{\an}$ with locally constant norm.
Corollary \ref{VFCCor} and  Remark \ref{QIMconnRmk} then show that  quantisations $S$ lead to virtual fundamental classes 
\[
 [e^S] \in \H^{BM}_{\dim X}(\pi^0X(\Cx)_{\an}, o_X)\brh
\]
in Borel--Moore homology of the underived truncation $\pi^0X$ of $X$ with its analytic topology,  where $\dim X$ is the virtual dimension of $X$.

In order to understand what these virtual fundamental classes look like, we then restrict attention to  dg manifolds $X$  with strict  $(-2)$-shifted symplectic structures, and relate them to Euler classes. Over $\Cx$, if we follow the procedure of \cite{BorisovJoyce} to cut the dg manifold in half, the resulting object $X_{\mathrm{dm}}$ is the derived vanishing locus of a section of a smooth real orthogonal  vector bundle. We then show in Proposition \ref{EulerPropNew} and Remark \ref{BJrmk} that the classes 
\[
[e^S] \in \H^{BM}_{\dim X}(\pi^0X(\Cx)_{\an}, o_X)\brh \cong \H^r(X^0(\Cx), X^0(\Cx)\setminus \pi^0X(\Cx); o_X)\brh
\]
associated to quantisations are given by pulling back Thom classes $e_r \in \H^r(B\OO_r, B\OO_{r-1};\det)$ (necessarily $0$ unless $\dim X$ is even) and multipling by elements of $1+ \hbar^2 \Cx\brh$.

In particular, when $X$ is oriented, this means (Corollary \ref{BJcor})  that the images in Steenrod homology of the classes $[e^S]$ are given by
\[
 \hbar^{(\dim X)/2} [X]_{BJ} \cdot (1+ \hbar^2 \Cx\brh) \subset \H_{\dim X}^{\mathrm{St}}(\pi^0X(\Cx)_{\an},\Cx\brh),
\]
where  $[X]_{BJ}$ is the Borisov--Joyce virtual fundamental class $[X_{\mathrm{dm}}]_{\mathrm{virt}}$ of \cite[Corollary 3.19]{BorisovJoyce}. 

It is important to note, however, that a quantisation itself is a far richer structure than a cohomology class, because of restrictions in terms of the Hodge filtration. In particular,  shifted Poisson structures can be recovered directly from our $E_{-1}$ quantisations, and the space of homotopy classes of quantisations  does not have an abelian structure.

I would like to thank Dominic Joyce for helpful comments on Borisov--Joyce invariants.

\subsection{Notation}

We denote the underlying  graded module of a cochain complex (resp. chain complex) by $M^{\#}$ (resp. $M_{\#}$).

Given a differential graded associative algebra (DGAA) $A$, and $A$-modules $M,N$ in cochain complexes,
we write $\HHom_A(M,N)$ for the cochain complex given by
\[
  \HHom_A(M,N)^i= \Hom(M^{\#},N^{\#[i]}),
\]
 with differential $\delta f= \delta_N \circ f \pm f \circ \delta_M$.

\tableofcontents

\section{Compatible quantisations on derived affine schemes}\label{affinesn}
Let $R$ be a graded-commutative differential algebra (CDGA) over $\Q$, and fix a CDGA $A$ over $R$. We will denote the differentials on $A$ and $R$  by $\delta$. 

\subsection{Quantised $(-2)$-shifted polyvectors}

\subsubsection{Polyvectors}

The following is adapted from \cite[Definition \ref{poisson-poldef}]{poisson}, with the introduction of a dummy variable $\hbar$ of cohomological degree $0$ to assist comparison with quantisation constructions.

\begin{definition}\label{poldef}
Define the complex of $(-2)$-shifted polyvector fields  (or strictly speaking, multiderivations) on $A$ by
\[
 \widehat{\Pol}(A/R,-2):=  \prod_{p \ge 0}\hbar^{p-1}\HHom_A(\Omega^p_{A},A)[p]. 
\]
with graded-commutative  multiplication $(a,b)\mapsto  ab$  on $ \hbar\widehat{\Pol}(A,-2)$   following the usual conventions for symmetric powers, so for $\pi \in  \hbar^p\HHom_A(\Omega^p_{A},A)$, $ \nu \in  \hbar^q\HHom_A(\Omega^q_{A},A)$ we have
\[
 (\pi\cdot \nu)(a df_1 \wedge \ldots \wedge df_{p+q})= \frac{1}{(p+q)!}\sum_{\sigma \in S_{p+q}} \pm a \pi( df_{\sigma(1)} \wedge \ldots \wedge df_{\sigma(p)}) \nu(df_{\sigma(p+1)} \wedge \ldots \wedge df_{\sigma(p+q)}). 
\]

The Lie bracket on $\Hom_A(\Omega^1_{A/R},A)$ then extends to give a bracket (the Schouten--Nijenhuis bracket)
\[
[-,-] \co \widehat{\Pol}(A/R,-2)\by \widehat{\Pol}(A/R,-2)\to \widehat{\Pol}(A/R,-2)[1],
\]
determined by the property that it is a bi-derivation with respect to the multiplication operation. 

Thus  $\widehat{\Pol}(A/R,-2)$ has the natural structure of a $P_0$-algebra, and in particular $\widehat{\Pol}(A/R,-2)[-1]$ is a differential graded Lie algebra (DGLA) over $R$.

Note that the differential $\delta$ on $\widehat{\Pol}(A/R,-2)$ can be written as $[\delta,-]$, where $\delta \in \widehat{\Pol}(A/R,-2)^1$ is the element defined by the derivation $\delta$ on $A$.
\end{definition}

\begin{definition}\label{Fdef}
Define a decreasing filtration $F$ on  $\widehat{\Pol}(A/R,-2)$ by 
\[
 F^i\widehat{\Pol}(A/R,-2):=  \prod_{j \ge i}\hbar^{j-1}\HHom_A(\Omega^j_{A},A)[j];
\]
this has the properties that $\widehat{\Pol}(A/R,-2)= \Lim_i \widehat{\Pol}(A/R,-2)/F^i$, with $[F^i,F^j] \subset F^{i+j-1}$, $\delta F^i \subset F^i$, and $ F^i F^j \subset \hbar^{-1} F^{i+j}$.
\end{definition}

Observe that this filtration makes $F^2\widehat{\Pol}(A/R,-2){[-1]}$ into a pro-nilpotent DGLA.

\begin{definition}\label{Tpoldef0}
 Define the tangent DGLA of polyvectors by 
\[
 T\widehat{\Pol}(A/R,-2):= \widehat{\Pol}(A/R,-2)\oplus  \widehat{\Pol}(A/R,-2) \hbar\eps,
\]
for $\eps$ of degree $0$ with $\eps^2=0$. The Lie bracket is given by $ [u+v\eps, x+y\eps]= [u,x]+ [u,y]\eps + [v,x]\eps$.
\end{definition}

\begin{definition}\label{Tpoldef}
Given a Maurer--Cartan element $\pi \in  \mc(F^2\widehat{\Pol}(A/R,-2){[-1]}) $, define 
\[
 T_{\pi}\widehat{\Pol}(A/R,-2):= \prod_{p \ge 0}\hbar^{p}\HHom_A(\Omega^p_{A},A)[p],
\]
with derivation $\delta + [\pi,-]$ (necessarily square-zero by the Maurer--Cartan conditions). 
The product on polyvectors makes this a  CDGA, and it inherits the filtration $F$ from $\widehat{\Pol}$. 

Given $\pi \in \mc(F^2\widehat{\Pol}(A/R,-2)/F^p)$, we define $T_{\pi}\widehat{\Pol}(A/R,-2)/F^p$ similarly. This is a CDGA  because $F^i\cdot F^j \subset F^{i+j}$.
\end{definition}

Regarding $ T_{\pi}\widehat{\Pol}(A/R,-2){[-1]}$ as an abelian DGLA, observe that $\mc(T_{\pi}\widehat{\Pol}(A/R,-2){[-1]})$ is just the fibre of 
\[
 \mc( T\widehat{\Pol}(A/R,-2){[-1]}) \to  \mc(\widehat{\Pol}(A/R,-2){[-1]})
\]
over $\pi$. Evaluation at $\hbar=1$ gives an  isomorphism from $T\widehat{\Pol}(A/R,-2){[-1]}$ to the DGLA   $\widehat{\Pol}(A/R,-2){[-1]}\ten_{\Q} \Q[\eps]$ of \cite[\S \ref{poisson-polsn}]{poisson}, and the map $\sigma$ of \cite[Definition \ref{poisson-sigmadef}]{poisson} then becomes:

\begin{definition}\label{sigmadef}
Define 
\[
\sigma =- \pd_{\hbar^{-1}} \co \widehat{\Pol}(A/R,-2) \to T\widehat{\Pol}(A/R,-2)
\]
 by $\alpha \mapsto \alpha + \eps \hbar^{2}\frac{\pd \alpha}{\pd \hbar}$. Note that this is a morphism of filtered  DGLAs, so gives a map 
\[
 \mc(F^2\widehat{\Pol}(A/R,-2){[-1]}) \to \mc(F^2T\widehat{\Pol}(A/R,-2){[-1]}),
\]
with $ \sigma(\pi) \in \z^1(F^2T_{\pi}\widehat{\Pol}(A/R,-2){[-1]})$. 
\end{definition}

\subsubsection{Right connections and de Rham complexes}

\begin{definition}\label{htpyrightDmoddef}
 We define a homotopy right $\sD$-module structure (or flat right connection) on $A$ over $R$ to be a sequence of maps $\nabla_{p+1} \co \HHom_A(\Omega^{p}_{A/R},A)^{\#} \to A^{\#}[1-p] $ for $p \ge 1$,  satisfying the following conditions:
\begin{enumerate}
 \item For $a \in A$ and $\xi \in  \HHom_A(\Omega^{1}_{A},A)$, we have $\nabla_2(a\xi)= a\nabla_2(\xi) - \xi(da)$;
\item For $p \ge 2$, the maps $\nabla_{p+1}$ are $A$-linear;
\item The operations $(\nabla_2 -\id, \nabla_3, \nabla_4, \ldots)$ define an $L_{\infty}$-morphism from the DGLA $ \HHom_A(\Omega^{1}_{A},A)$ to the DGLA $(A \oplus  \HHom_A(\Omega^{1}_{A},A))^{\op}$ of first-order differential operators with bracket given by negating the commutator.
\end{enumerate}
\end{definition}

\begin{remarks}
The final condition in Definition \ref{htpyrightDmoddef} is equivalent to saying that $\nabla$ is an $L_{\infty}$-derivation from the DGLA $ \HHom_A(\Omega^{1}_{A},A)$ to the $R$-module   $A$ given the $ \HHom_A(\Omega^{1}_{A},A)$-module structure $\xi * a := -\xi(da)$.  
 If we interchange the order of duals and tensor products (permissible if $\Omega^1_A$ is a perfect $A$-module), then our flat right connections correspond to  right $(A,\HHom_A(\Omega^{1}_{A},A))$-module structures on $A$, in the sense of \cite[Definition 44]{vitagliano}, for the natural Lie-Rinehart algebra structure on $(A,\HHom_A(\Omega^{1}_{A},A))$. 

\end{remarks}

\begin{definition}\label{DRrdef}
 Given a flat right connection $\nabla$ on $A$, we define the right de Rham complex $\DR^r(A,\nabla)$ associated to $\nabla$, and its increasing filtration $F$, by 
\[
 F_i \DR^r(A,\nabla):= \bigoplus_{p \le i}\HHom_A(\Omega^p_{A},A)[p], 
\]
equipped with differential $D^{\nabla}= \sum_{k \ge 1} D^{\nabla}_k$ 
 given (for $\pi \in\HHom_A(\Omega^p_{A},A)[p]$, $\omega \in \Omega^{p+1-k}_A$)  by
\begin{align*}
 D^{\nabla}_k(\pi)(\omega):= \begin{cases}
                              \nabla_k(\pi \lrcorner \omega) & k > 2,\\
                           \nabla_2(\pi \lrcorner \omega) +(-1)^{\deg \pi} \pi(d\omega) & k=2,\\
\delta\pi( \omega)  &k=1,
\end{cases}
\end{align*}
where $d$ is the de Rham differential and $\delta$ is induced by the differential $\delta$ on $A$.
\end{definition}

As in \cite[Corollary 50]{vitagliano}, the condition that $\nabla$ be an $L_{\infty}$-derivation is equivalent to saying that the operator $ D^{\nabla}= \sum_k D^{\nabla}_k$ satisfies $D^{\nabla} \circ D^{\nabla}=0$.




\begin{definition}\label{BVdef}
We adapt 
\cite[Definition 7]{kravchenko} 
by defining a filtered $BV_{\infty}$-algebra $B$ over $R$ to be a graded-commutative unital $R$-algebra equipped with an increasing filtration $F$  and a  square-zero $R$-linear operator $\boxempty$ of degree $1$ satisfying the conditions
\begin{enumerate}
 \item $1 \in F_0B$ and  $F_r \cdot F_s \subset F_{r+s}$;
\item $\boxempty(1)=0$ and $\boxempty(F_r) \subset F_r$;
\item  for $a_i \in F_{r_i}$ and $b \in F_s$, the iterated graded commutators satisfy
\[
 [a_1,[a_2, \ldots, [a_k, \boxempty] \ldots ](b) \in F_{s-k+\sum r_i}. 
\]
\end{enumerate}
\end{definition}
In particular, the  conditions for the  BV operator $\boxempty$ are satisfied if it admits a locally finite decomposition $\boxempty = \sum_{k\ge 1} \boxempty_k$, for $\boxempty_k$ a differential operator of  order $\le k$, such that $\boxempty_k(1)=0$ and $\sum_{i+j=k}[\boxempty_i,\boxempty_j]=0$, with $\boxempty_k\co F_rB \to F_{r+1-k}B$.  Such decompositions without a filtration correspond to $BV_{\infty}$-algebras in the sense of  \cite[Definition 52]{vitagliano} and \cite[Definition 3.11]{BraunLazarevHtpyBV}.

\begin{definition}\label{LinftyBVdef}
Following \cite[Proposition 2]{kravchenko}, the operations
\[
 [a_1, \ldots, a_k]_{\boxempty,k}:= [\ldots [\boxempty, a_1], \ldots , a_k](1)
\]
define an $L_{\infty}$-algebra structure on the complex $B[-1]$ for  any filtered $BV_{\infty}$-algebra  $B$ with differential $\boxempty$.

It follows from Definition \ref{BVdef} that these $L_{\infty}$ operations  satisfy
\[
 [F_{i_1}, \ldots, F_{i_k}]_{\boxempty,k}\subset F_{i_1+ \ldots +i_k +1-k}.
\]
\end{definition}

\begin{lemma}\label{DRrBVlemma}
 The operator $D^{\nabla}= \sum_{k \ge 1} D^{\nabla}_k $ defines a filtered   $BV_{\infty}$-algebra structure on the filtered  graded-commutative algebra $\DR^r(A,\nabla)$. 
On the associated graded complex $\gr^F\DR^r(A,\nabla) $,  the induced $L_{\infty}$  bracket $[-]_{\nabla,k}$ of weight $1-k$ 
is trivial for $k \ge 3$ and corresponds to the Schouten--Nijenhuis bracket for $k=2$.
\end{lemma}
\begin{proof}
 The argument of \cite[Proposition 53]{vitagliano} adapts to our slightly different setting to show that the operators $D^{\nabla}_k$ define a $BV_{\infty}$-algebra structure.  It follows directly from the definitions that $F_i \cdot F_j \subset F_{i+j}$ and  that $D^{\nabla}_k(F_i) \subset F_{i+1-k}$

Now, differential operators of order less than $k$ do not contribute to the  $(k-1)$-fold commutator $[-]_{\nabla,k}$, so $[-]_{\nabla,k}= \sum_{j \ge k} [-]_{\nabla_j,k}$, which is of weights at most $(1-k)$ with respect to the filtration. Observe that the leading term of $[-]_{\nabla,k}$ is $[-]_{\nabla_k,k}$. The calculation of \cite[Proposition 53]{vitagliano} shows that this structure corresponds to the $L_{\infty}$-structure on polyvectors induced by the $L_{\infty}$-structure on $\HHom_A(\Omega^{1}_{A/R},A)$, which is just the Schouten--Nijenhuis bracket.
\end{proof}

\subsubsection{Quantised  polyvectors}\label{qpolsn}

\begin{definition}
Given a flat right connection $\nabla$ on $A$, define the complex 
of quantised $(-2)$-shifted polyvector fields  on $A$ by
\[
 Q\widehat{\Pol}(A,\nabla,-2):= \prod_j \hbar^{j-1}F_j\DR^r(A,\nabla).
\]
\end{definition}
It follows from Lemma \ref{DRrBVlemma} that   for $a_i, b \in \hbar Q\widehat{\Pol}(A,\nabla,-2)$, the iterated graded commutators satisfy
\[
 [a_1,[a_2, \ldots, [a_k,D^{\nabla}] \ldots ](b) \in \hbar^{k+1} Q\widehat{\Pol}(A,\nabla,-2), 
\]
and that  $\hbar Q\widehat{\Pol}(A,\nabla,-2)$ is closed under multiplication. Thus multiplication and the operator $D^{\nabla}$  make $\hbar Q\widehat{\Pol}(A,\nabla,-2) $ into a  filtered $BV_{\infty}$-algebra
with respect to the 
$\hbar$-adic filtration.
Moreover, the induced $L_{\infty}$-algebra structure from  Definition \ref{LinftyBVdef}  extends naturally to an $R\brh$-linear $L_{\infty}$-algebra structure on   $Q\widehat{\Pol}(A,\nabla,-2)[-1]$.

\begin{definition}\label{QFdef}
Define a decreasing filtration $\tilde{F}$ on  $Q\widehat{\Pol}(A,\nabla,-2)$ by 
\[
 \tilde{F}^iQ\widehat{\Pol}(A,\nabla,-2):= \prod_{j \ge i} \hbar^{j-1}F_j\DR^r(A,\nabla).
\]
\end{definition}
This filtration has the properties that $Q\widehat{\Pol}(A,\nabla,-2)= \Lim_i Q\widehat{\Pol}(A,\nabla,-2)/\tilde{F}^i$, 
with multiplication in $\DR^r$ giving us a commutative product
\[
 \tilde{F}^iQ\widehat{\Pol}(A,\nabla,-2) \by \tilde{F}^jQ\widehat{\Pol}(A,\nabla,-2) \to \hbar^{-1} \tilde{F}^{i+j}Q\widehat{\Pol}(A,\nabla,-2).
\]
The operators $D^{\nabla}_k$ satisfy
 $D^{\nabla}_k (\tilde{F}^i) \subset \hbar^{k-1} \tilde{F}^{i+1-k} \subset\tilde{F}^i  $ and 
\[
[\tilde{F}^{i_1}, \ldots \tilde{F}^{i_m}]_{\nabla,m} \subset \tilde{F}^{i_1+ \ldots +i_m +1-m}.
\]

\subsection{$(-2)$-shifted quantisations}

\subsubsection{The space of quantisations}

\begin{definition}\label{mcPLdef}
 Given an $L_{\infty}$-algebra $L$,  the Maurer--Cartan set is defined by 
\[
\mc(L):= \{\omega \in  L^{1}\ \,|\, \sum_{n \ge 1} \frac{[\omega, \ldots, \omega]_n}{n!}  =0 \in   L^{2}\},
\]
where $[-]_1$ is the differential.

Following \cite{hinstack}, define the Maurer--Cartan space $\mmc(L)$ (a simplicial set) of 
$L$ by
\[
 \mmc(L)_n:= \mc(L\ten_{\Q} \Omega^{\bt}(\Delta^n)),
\]
with the obvious simplicial operations, where 
\[
\Omega^{\bt}(\Delta^n)=\Q[t_0, t_1, \ldots, t_n,\delta t_0, \delta t_1, \ldots, \delta t_n ]/(\sum t_i -1, \sum \delta t_i)
\]
is the commutative dg algebra of de Rham polynomial forms on the $n$-simplex, with the $t_i$ of degree $0$.
\end{definition}
%

\begin{definition}\label{Gdef}
We now define another decreasing filtration $G$ on  $Q\widehat{\Pol}(A, \nabla,-2)$ by setting
\[
 G^iQ\widehat{\Pol}(A, \nabla,-2):= \hbar^{i}Q\widehat{\Pol}(A, \nabla,-2).
\]
We then set $G^i \tilde{F}^p:= G^i \cap \tilde{F}^p$.
\end{definition}

 Note that $G^i \subset \tilde{F}^i$, and  beware that $G^i \tilde{F}^p$ is not the same as $\hbar^{i} \tilde{F}^p$ in general, since 
\begin{align*}
 G^i\tilde{F}^pQ\widehat{\Pol}(A, \nabla,-2) &= \prod_{j \ge p} \hbar^{j-1}F_{j-i}\DR^r(A,\nabla),\\
\hbar^{i}\tilde{F}^pQ\widehat{\Pol}(A, \nabla,-2) &= \prod_{j \ge p+i}\hbar^{j-1} F_{j-i}\DR^r(A,\nabla).
\end{align*}

We will also consider the convolution  $G*\tilde{F}$, given by $(G*\tilde{F})^p:= \sum_{i+j=p}G^i\cap\tilde{F}^j$ ; explicitly,
\[
(G*\tilde{F})^p Q\widehat{\Pol}(A, \nabla,-2) = 
\prod_{j<p}\hbar^{j-1}F_{2j-p}\DR^r(A,\nabla) \by \prod_{j \ge p}\hbar^{j-1} F_j\DR^r(A,\nabla).
\]
In particular, $(G*\tilde{F})^2 Q\widehat{\Pol}(A, \nabla,-2) = A \oplus \tilde{F}^2 Q\widehat{\Pol}(A, \nabla,-2)$.


\begin{definition}\label{Qpoissdef}
Define  the space $Q\cP(A, \nabla,-2)$ of $E_{-1}$ quantisations of $(A,\nabla)$ over $R$  to be given by the simplicial 
set
\[
 Q\cP(A, \nabla,-2):= \Lim_i \mmc( \tilde{F}^2 Q\widehat{\Pol}(A, \nabla,-2)[-1]/\tilde{F}^{i+2}).
\]
Also write
\[
 Q\cP(A, \nabla,-2)/G^k:= \Lim_i\mmc(\tilde{F}^2 Q\widehat{\Pol}(A, \nabla,-2)[-1]/(\tilde{F}^{i+2}+G^k)),
\]
 so $Q\cP(A, \nabla,-2)= \Lim_k Q\cP(A, \nabla,-2)/G^k$. 
\end{definition}

\subsubsection{The quantum master equation}\label{QME}


By  \cite[Theorem 3.7]{BraunLazarevHtpyBV}, there is an $L_{\infty}$-isomorphism from $\DR^r(A,\nabla)[-1] $ with the $L_{\infty}$-structure $[-]_{\nabla} $ of Definition \ref{LinftyBVdef} to the complex $\DR^r(A,\nabla)[-1] $ with abelian $L_{\infty}$ structure. Applied to the pro-nilpotent $L_{\infty}$-algebra $\hbar\DR^r(A,\nabla)\brh[-1]$, this gives an isomorphism
\begin{align*}
 \Lim_r \mmc(\hbar(\DR^r(A,\nabla)[\hbar]/\hbar^r)[-1]; [-]_{\nabla}) &\to \Lim_r \mmc(\hbar(\DR^r(A,\nabla)[\hbar]/\hbar^r)[-1]; D^{\nabla}, 0, 0,\ldots)\\
S \mapsto e^{S}-1.
\end{align*}
In particular, for $S \in \DR^r(A,\nabla)^0$, \cite[Remark 3.6]{BraunLazarevHtpyBV} shows that the expression  $\sum_n[S, \ldots,S]_{n, \nabla}/n!$ can be rewritten as 
$e^{-S}D_{S}^{\nabla}(e^S)$, so the  Maurer--Cartan equation $\sum_n[S, \ldots,S]_{n, \nabla}/n!=0$ is equivalent to the quantum master equation $D^{\nabla}(e^{S})=0$.

Since the  target $L_{\infty}$-algebra is abelian,  $\Lim_r \mmc(\hbar(\DR^r(A,\nabla)[\hbar]/\hbar^r)[-1]; D^{\nabla}, 0, 0,\ldots)$    should be thought of as the space of $0$-cocycles in the right de Rham complex $\hbar\DR^r(A,\nabla)\brh$. Its homotopy groups are given by
\[
 \pi_i \Lim_r \mmc(\hbar(\DR^r(A,\nabla)[\hbar]/\hbar^r)[-1]; D^{\nabla}, 0, 0,\ldots) \cong \hbar \H^{-i}(\DR^r(A,\nabla))\brh.
\]
A smaller, but weakly equivalent space can be constructed by truncating the complex $\DR^r(A,\nabla)$ in non-positive degrees, and applying the inverse of the Dold--Kan normalisation functor to obtain a simplicial abelian group.

Our complex  $Q\widehat{\Pol}(A, \nabla,-2)$ is not itself a $BV_{\infty}$-algebra, but it is an $L_{\infty}$-subalgebra of $(\hbar(\DR^r(A,\nabla)[\hbar]/\hbar^r)[-1]; [-]_{\nabla})$. Therefore  sending $S$ to $e^S-1$ gives natural maps
\begin{align*}
  Q\cP(A, \nabla,-2) &\to \Lim_r \mmc(\hbar(\DR^r(A,\nabla)[\hbar]/\hbar^r)[-1]; D^{\nabla}, 0, 0,\ldots),\\
  \pi_iQ\cP(A, \nabla,-2) &\to \hbar \H^{-i}(\DR^r(A,\nabla))\brh
\end{align*}
from quantisations to power series 
in right de Rham cohomology. This will lead to comparisons with other constructions in \S \ref{cfBJ}.

\subsubsection{The centre of a quantisation}\label{centresn}

\begin{definition}\label{TQpoldef0}
 Define the filtered tangent $L_{\infty}$-algebra of quantised polyvectors by 
\begin{align*}
 TQ\widehat{\Pol}(A, \nabla,-2)&:= Q\widehat{\Pol}(A, \nabla,-2)\oplus \prod_{p \ge 0}\hbar^{p}F_p\DR^r(A,\nabla)\eps,\\
\tilde{F}^jTQ\widehat{\Pol}(A,\nabla,-2)&:= \tilde{F}^jQ\widehat{\Pol}(A,\nabla,-2)\oplus \prod_{p \ge j}\hbar^{p}F_p\DR^r(A,\nabla)(M)\eps,
\end{align*}
for $\eps$ of degree $0$ with $\eps^2=0$. The $L_{\infty}$ operations are  given by $ [u_1+v_1\eps, \ldots,  u_n+v_n\eps]_n= [u_1,\ldots, u_n]_n+ \sum_{i=1}^n   [u_1,\ldots, u_{i-1}, v_i, u_{i+1}, \ldots,u_n]\eps$.
\end{definition}

\begin{definition}\label{TQpoldef}
Given a Maurer--Cartan element $S \in  \mc(\tilde{F}^2Q\widehat{\Pol}(A, \nabla,-2)[-1]) $, define the centre of $(A, \nabla,S)$ by 
\[
 T_{S}Q\widehat{\Pol}(A, \nabla,-2):= \prod_{p \ge 0}\hbar^{p}F_p\DR^r(A,\nabla),
\]
with differential $D^{\nabla}_S =e^{-S} \circ D^{\nabla} \circ e^S$ (necessarily square-zero). 
 
This inherits a commutative multiplication from $\DR^r(A,\nabla)$, and it has a filtration
\[
 \tilde{F}^iT_{S}Q\widehat{\Pol}(A, \nabla,-2):= \prod_{p \ge i}\hbar^{p}F_p\DR^r(A,\nabla),
\]
with $\tilde{F}^i \cdot \tilde{F}^j\subset \tilde{F}^{i+j}$. 

Given $S \in  \mc(F^2Q\widehat{\Pol}(A, \nabla,-2)[-1]/\tilde{F}^p)$, we define $T_{S}Q\widehat{\Pol}(A, \nabla,-2)/\tilde{F}^p$ similarly.
\end{definition}

Observe that regarding $T_{S}Q\widehat{\Pol}(A, \nabla,-2) $ as an abelian $L_{\infty}$-algebra,  the space
\[
T_{S}Q\cP(A, \nabla,-2)/\tilde{F}^p:= \mmc(\tilde{F}^2T_{S}Q\widehat{\Pol}(A, \nabla,-2)[-1]/\tilde{F}^p;D^{\nabla}_S, 0, 0,\ldots )
\]
 is just the fibre of 
\[
\mmc( \tilde{F}^2TQ\widehat{\Pol}(A, \nabla,-2)[-1]/\tilde{F}^p) \to  \mmc(\tilde{F}^2 Q\widehat{\Pol}(A, \nabla,-2)[-1]/\tilde{F}^p)
\]
 over $S$.

Similarly to Definition \ref{Gdef}, there  is a filtration $G$ on  
$TQ\widehat{\Pol}(A, \nabla,-2)$ and  $T_{S}Q\widehat{\Pol}(A, \nabla,-2) $ given by powers of $\hbar$. 
This filtration makes $T_{S}Q\widehat{\Pol}(A, \nabla,-2) $ into  a filtered $BV_{\infty}$-algebra in the sense of Definition \ref{BVdef}, and  $T_{S}Q\widehat{\Pol}(A, \nabla,-2)/\tilde{F}^p$ is then
 also a filtered $BV_{\infty}$-algebra (with respect to the filtration $G$) since $\tilde{F}^p$ is an ideal.

Since $\gr_G^i\tilde{F}^{p-i}Q\widehat{\Pol}= \prod_{j \ge p-i} \hbar^{j-1}\gr^F_{j-i\DR^r(A,\nabla)}$,
the associated gradeds of the filtration $G$ admit maps 
  \begin{align*}
 \gr_G^i\tilde{F}^pQ\widehat{\Pol}(A, \nabla,-2) &\to \prod_{j \ge p}  \hbar^{j-1}\HHom_A(\Omega^{j-i}_{A/R},A)[j-i]\\
\gr_G^i\tilde{F}^pT_{S}Q\widehat{\Pol}(A,\nabla,-2) &\to \prod_{j \ge p} \hbar^{j} \HHom_A(\Omega^{j-i}_{A/R},A)[j-i].
\end{align*}
which are isomorphisms when $A$ is  an affine dg manifold in the sense of \cite[Definition 2.5.1]{Quot}; in particular whenever $A$ is cofibrant as a CDGA over $R$.

For the filtrations $F$ of Definitions \ref{Fdef} and \ref{Tpoldef}, we may rewrite these maps as 
 \begin{align*}
 \gr_G^i\tilde{F}^pQ\widehat{\Pol}(A, \nabla,-2) &\to  \hbar^{i}F^{p-i}\widehat{\Pol}(A, -2),\\ 
\gr_G^i\tilde{F}^pT_{S}Q\widehat{\Pol}(A, \nabla,-2) &\to \hbar^{i}F^{p-i}T_{\pi_{S}}\widehat{\Pol}(A, -2),
\end{align*}
where $\pi_{S} \in \mc(F^2\widehat{\Pol}(A, \nabla,-2))$ denotes the image of $S$ under the map  $ \gr_G^0\tilde{F}^2Q\widehat{\Pol}(A, \nabla,-2) \to  F^2\widehat{\Pol}(A,-2)$. 

Since the cohomology groups of $T_{\pi_{S}}\widehat{\Pol}(A,-2)$ are shifted Poisson  cohomology, we will refer to the cohomology groups of  $T_{S}Q\widehat{\Pol}(A, \nabla,-2)$ as quantised Poisson cohomology.

We write $Q^{tw}\cP(A, \nabla,-2):= \mmc((G*\tilde{F})^2Q\widehat{\Pol}(A, \nabla,-2))$ and   $T_{S}Q^{tw}\cP(A, \nabla,-2):= \mmc((G*\tilde{F})^2T_{S}Q\widehat{\Pol}(A, \nabla,-2))$.

\begin{definition}\label{Qnonnegdef}
 Say that an $E_{-1}$ quantisation $S =  \sum_{j \ge 2} S_j \hbar^{j}$ is non-degenerate if the map 
\[
S_2^{\sharp}\co   \Omega^1_A \to \HHom_A(\Omega^1_A,A)[2]
\]
is a quasi-isomorphism and $\Omega^1_A $ is perfect.
\end{definition}

\begin{definition}\label{TQPdef}
Define the tangent spaces of quantisations and of twisted quantisations by
\begin{eqnarray*}
 TQ\cP(A, \nabla,-2)&:=& \Lim_i \mmc( \tilde{F}^2 TQ\widehat{\Pol}(A, \nabla,-2)/\tilde{F}^{i+2}),\\
TQ^{tw}\cP(A, \nabla,-2)&:=& \Lim_i \mmc( (G*\tilde{F})^2 TQ\widehat{\Pol}(A, \nabla,-2)/\tilde{F}^{i+2}),
\end{eqnarray*}
with $TQ\cP(A, \nabla,-2)/G^k$, $TQ^{tw}\cP(A, \nabla,-2)/G^k $ defined similarly.
 \end{definition}
These are simplicial sets over $Q\cP(A, \nabla,-2)$ (resp.  $Q^{tw}\cP(A, \nabla,-2)$, $Q\cP(A, \nabla,-2)/G^k$, $Q^{tw}\cP(A, \nabla,-2)/G^k$), fibred in simplicial abelian groups.  

\begin{definition}\label{Qsigmadef}
Define the canonical tangent vector
\[
  \sigma=-\pd_{\hbar^{-1}}\co Q\widehat{\Pol}(A, \nabla,-2) \to TQ\widehat{\Pol}(A, \nabla,-2)
\]
 by $\alpha \mapsto \alpha + \eps \hbar^{2}\frac{\pd \alpha}{\pd \hbar}$. Note that this is a morphism of filtered  DGLAs, so gives a map $ \sigma \co Q\cP(A, \nabla,-2)\to TQ\cP(A, \nabla,-2)$, with $\sigma(S) \in S+\eps \z^1(\tilde{F}^2T_{S}Q\widehat{\Pol}(A, \nabla,-2))$ for $S \in Q\cP(A, \nabla,-2)_0$. 
\end{definition}

\subsection{Generalised pre-symplectic structures}\label{DRsn}

\begin{definition}
Define the (left) de Rham complex $\DR(A/R)$ to be the product total complex of the bicomplex
\[
 A \xra{d} \Omega^1_{A/R} \xra{d} \Omega^2_{A/R}\xra{d} \ldots,
\]
so the total differential is $d \pm \delta$.

We define the Hodge filtration $F$ on  $\DR(A/R)$ by setting $F^p\DR(A/R) \subset \DR(A/R)$ to consist of terms $\Omega^i_{A/R}$ with $i \ge p$. In particular, $F^p\DR(A/R)= \DR(A/R)$ for $p \le 0$.
\end{definition}

\begin{definition}
When $A$ is a  CDGA over $R$ which is an affine dg manifold, recall that  a $(-2)$-shifted pre-symplectic structure $\omega$ on $A/R$ is an element
\[
 \omega \in \z^{0}F^2\DR(A/R).
\]
In \cite{PTVV}, shifted pre-symplectic structures are referred to as closed $2$-forms.

A $(-2)$-shifted pre-symplectic structure $\omega$ is called symplectic if $\omega_2 \in \z^{-2}\Omega^2_{A/R}$ induces a quasi-isomorphism
\[
 \omega_2^{\sharp} \co \Hom_A(\Omega^1_{A/R},A) \to \Omega^1_{A/R}[-2],
\]
and $\Omega^1_{A/R} $ is perfect as an $A$-module.
\end{definition}

We now recall a construction from \cite[Definition \ref{DQvanish-DRprimedef}]{DQvanish} which allows us to formulate compatibility between quantisations and a generalisation of pre-symplectic structures. 

\begin{definition}\label{DRprimedef}
Write $A^{\ten \bt +1}$ for the cosimplicial CDGA $n \mapsto A^{n +1}$ given by the \v Cech nerve, with $I$ the kernel of the diagonal map $A^{\ten \bt +1}\to A$. This has a filtration $F$ given by powers $F^p :=(I)^p$ of $I$, and we define the filtered cosimplicial CDGA $\hat{A}^{\ten \bt +1}$ to be the completion
\begin{align*}
 \hat{A}^{\ten \bt +1}&:= \Lim_qA^{\ten \bt +1}/F^q,\\
F^p \hat{A}^{\ten \bt +1}&:= \Lim_qF^p/F^q.
\end{align*}

We then take the Dold--Kan conormalisation $N\hat{A}^{\bt +1}$, which becomes a filtered bi-DGAA via the  Alexander--Whitney cup product. Explicitly,  $N^n \hat{A}^{\bt +1}$ is the intersection of the kernels of all the big diagonals $  \hat{A}^{n +1}\to \hat{A}^{n}$, and the cup  product is given by 
\[
 (a_0 \ten \ldots \ten a_m) \smile (b_0 \ten \ldots \ten b_n)= a_0 \ten \ldots \ten a_{m-1}\ten (a_mb_0) \ten b_1 \ten \ldots \ten b_n. 
\]

 We then define $\DR'(A/R)$ to be the product total complex
\[
 \DR'(A/R):=\Tot^{\Pi} N\hat{A}^{\bt +1}
\]
regarded as a filtered DGAA over $R$, with $F^p\DR'(A/R):=\Tot^{\Pi} NF^p\hat{A}^{\bt +1}$.
\end{definition}

The following is standard (for instance \cite[Lemma \ref{DQvanish-cfDRlemma}]{DQvanish}):
\begin{lemma}\label{cfDRlemma}
 There is a filtered quasi-isomorphism $\DR'(A/R) \to \DR(A/R)$, given by $N^n\hat{A}^{\bt +1}\to 
N^n\hat{A}^{\bt +1}/F^{n+1} \cong (\Omega^1_{A/R})^{\ten_A n} \to  \Omega^n_{A/R}$. 
\end{lemma}

\begin{definition}\label{tildeFDRdef}
Define a decreasing filtration $\tilde{F}$ on $ \DR'(A/R)\llbracket\hbar\rrbracket$ by 
\[
 \tilde{F}^p\DR'(A/R):= \prod_{i\ge 0} \hbar^{i}F^{p-i}\DR'(A/R),
\]
where we adopt the convention that $F^j\DR'=\DR'$ for all $j\le 0$.

Define  further filtrations $G, G*\tilde{F}$ by $ G^k \DR'(A/R)\llbracket\hbar\rrbracket = \hbar^{k}\DR'(A/R)\llbracket\hbar\rrbracket$, and $(G*\tilde{F})^p:= \sum_{i+j=p}G^i\cap\tilde{F}^j$, so
\[
 (G*\tilde{F})^p= \prod_{i\ge 0}\hbar^{i} F^{p-2i}\DR'(A/R).
\]
\end{definition}

This makes  $(\DR'(A/R)\llbracket\hbar\rrbracket,G*\tilde{F}) $ into a filtered DGAA, since $\tilde{F}^p\tilde{F}^q \subset \tilde{F}^{p+q}$ and similarly for $G$.

\begin{definition}
Define a generalised $(-2)$-shifted pre-symplectic structure on a cofibrant CDGA (or just an affine dg manifold) $A/R$ to be an element
\[
 \omega \in \z^0((G*\tilde{F})^2\DR'(A/R)\llbracket\hbar\rrbracket) = \z^0(F^2\DR'(A/R)) \by \hbar \z^0\DR'(A/R)\llbracket\hbar\rrbracket.
\]
Call this symplectic if $\Omega^1_{A/R}$ is perfect as an $A$-module and the  leading term $\omega_0 \in \z^0F^2\DR'(A/R)$ induces a quasi-isomorphism
\[
 [\omega_0]^{\sharp} \co \Hom_A(\Omega^1_{A/R},A) \to \Omega^1_{A/R}[-2],
\]
for $[\omega_0] \in \z^{-2}\Omega^2_{A/R}$ the image of $\omega_0$ modulo $F^3$.
\end{definition}

\begin{definition}\label{GPreSpdef}
 Define the space of generalised $(-2)$-shifted pre-symplectic structures on $A/R$ to be the simplicial set
\[
 G\PreSp(A/R,-2):= \Lim_i\mmc( (G*\tilde{F})^2\DR'(A/R)\llbracket\hbar\rrbracket[-1]/\tilde{F}^{i+2}), 
\]
where we regard the cochain complex  $\DR'(A/R)$ as a  DGLA with trivial bracket.

Also write $G\PreSp(A/R,-2)/\hbar^{k}:= \Lim_i\mmc( ((G*\tilde{F})^2\DR'(A/R)[\hbar]/(G^k +\tilde{F}^{i+2)} ))$, so $ G\PreSp(A/R,-2)= \Lim_k G\PreSp(A/R,-2)/\hbar^{k} $. Write $\PreSp = G\PreSp/\hbar$.

Set $G\Sp(A/R,-2) \subset G\PreSp(A/R,-2)$ to consist of the symplectic structures --- this is a union of path-components.
\end{definition}
Note that  $G\PreSp(A/R,-2)$  is canonically weakly  equivalent to the Dold--Kan denormalisation of the good truncation complex $\tau^{\le 0}((G*\tilde{F})^2\DR(A/R)\llbracket\hbar\rrbracket)$ (and similarly for the various quotients we consider),   but the description in terms of $\mmc$ will simplify comparisons. In particular, we have
\[
 \pi_iG\PreSp(A/R,-2)\cong \H^{-i}(F^2\DR(A/R)) \by \hbar \H^{-i}(\DR(A/R))\llbracket\hbar\rrbracket. 
\]

\subsection{Compatibility of quantisations and symplectic structures}

We will now develop the notion of compatibility between a (truncated) generalised $(-2)$-shifted pre-symplectic structure and a (truncated) $E_{-1}$  quantisation. The case $k=1$ recovers the notion of compatibility between $(-2)$-shifted pre-symplectic and Poisson structures from \cite{poisson}. From now on we fix a  CDGA $A$ over $R$ for which $(\Spec A_0, A)$ is an affine dg manifold.
%
%

\begin{proposition}\label{QPolmudef}
 Given  $S \in ((G*\tilde{F})^2Q\widehat{\Pol}(A,\nabla,-2)/G^k)^{1}$, there is a  chain map
\[
 \mu(-,S) \co \DR'(A/R)[\hbar]/G^k \to T_SQ\widehat{\Pol}(A,\nabla,-1)/G^k
\]
of graded associative $R[\hbar]/\hbar^{k}$-modules, respecting the filtrations $(G*\tilde{F})$;
this is  induced by the maps 
\[
a_0 \ten a_1 \ten \ldots \ten a_n \mapsto a_0 (D_S^{\nabla}-\delta)(a_1 (D_S^{\nabla}-\delta)(\ldots (D^{\nabla}_S-\delta)(a_n)\ldots )) 
\]
on $A^{\ten \bt +1}$.

Given $\rho \in ((G*\tilde{F})^pQT_S\widehat{\Pol}(A,\nabla,-2)/G^k)^r$, there is then an $R[\hbar]/\hbar^{k}$-linear derivation
\[
 \nu(-, S, \rho) \co (\DR'(A/R)[\hbar]/\hbar^{k},(G*\tilde{F})^{\bt})  \to (T_SQ\widehat{\Pol}(A,\nabla,-2)[r]/G^k,(G*\tilde{F})^{\bt+p}),
\]
which is characterised by the expression
\[
 \mu(\omega, S+ \eps \rho)= \mu(\omega,S) + \eps \nu(\omega, S,\rho),
\]
where $\eps^2=0$.
\end{proposition}
\begin{proof}
We adapt the proof of \cite[Lemma \ref{DQvanish-QPolmudef}]{DQvanish}. It suffices to prove this for the limit over all $k$, as $S$ and $\rho$ always lift to $(G*\tilde{F})^2Q\widehat{\Pol}(A,\nabla,-2)$ and the maps are $R\brh$-linear. First, let 
$T= T_0Q\widehat{\Pol}(A,\nabla,-2)[\hbar^{-1}]$, with filtrations $\tilde{F}$ given by powers of $\hbar$  and $G^iT:= \hbar^i T_0Q\widehat{\Pol}(A,\nabla,-2)$. We then consider the convolution filtration  $G*\tilde{F}$, which is given explicitly by $(G*\tilde{F})^pT= \prod_k \hbar^k F_{2k-p}\DR^r(A,\nabla)$.

The filtration $G*\tilde{F}$ induces a filtration on the ring $\sD_{R\brh}(T)$  of graded $R\brh$-linear  differential operators on  $T$, which we also denote by $G*\tilde{F}$, given by saying that 
\[
 (G*\tilde{F})^i\EEnd_{R\brh}(T) =\{ f \in \EEnd_{R\brh}(T) ~:~ f((G*\tilde{F})^pT) \subset (G*\tilde{F})^{i+p}T ~\forall p\}.
\]
We then let $B$ be the completion of $\sD_{R\brh}(T)$ with respect to the filtration  $G*\tilde{F}$, and define  a filtration $\Fil$ on $\sD_{R\brh}(T)$ by first setting 
\[
\Fil^i(\sD/(G*\tilde{F})^j)=\sum_{r} (G*\tilde{F})^{r}(\sD_{\le 2r-i}/(G*\tilde{F})^j),
\]
 where $\sD_{\le r}$ denotes differential operators of order $\le r$, then by letting  $\Fil^i(B)=\Lim_j\Fil^i(\sD/(G*\tilde{F})^j) $. 

Now, we have $D^{\nabla}=\sum_{k \ge 1} D_k^{\nabla}$, where the operator $D_k^{\nabla}$ has order $\le k$, preserves $\tilde{F}$ and shifts the index of $G$ by $k-1$. Thus $D_k^{\nabla} \in (G*\tilde{F})^{k-1}\sD_{\le k}$, so $D_k^{\nabla} \in \Fil^0B$ for $k \ge 2$ 
 and hence $D^{\nabla}-\delta  \in \Fil^0B$. 

Moreover, $\tilde{F}^pQ\widehat{\Pol}(A,\nabla,-2)=  \tilde{F}^{p-1}G^{-1}T $, so in particular $S \in  \tilde{F}^{1}G^{-1}T \subset (G*\tilde{F})^0T$. If we regard $S$ as a differential operator on $T$ of order $0$, this also gives $S \in \Fil^0B$. Since 
\[
 D^{\nabla}_S= e^{-S}D^{\nabla}e^S= e^{-S}(D^{\nabla}-\delta) e^S + (\delta + \delta S),
\]
it follows that $D^{\nabla}_S-\delta  \in \Fil^0B$.

Since the associated graded ring $\gr_{\Fil}^{\bt}B$ is commutative, we may therefore appeal to \cite[Lemma \ref{DQvanish-mulemma1}]{DQvanish}, which gives a filtered map 
$
\mu'(-,S) \co (\DR'(A/R),F) \to (B,\Fil),
$
from which we obtain our filtered morphisms
\begin{align*}
 \mu(-,S) \co (\DR'(A/R),F^{\bt}) &\to (T,G*\tilde{F}^{\bt})\\
 \nu(-, S, \rho) \co(\DR'(A/R),F^{\bt}) &\to (T[r], (G*\tilde{F})^{\bt+p})
\end{align*}
by evaluating the operators at $1$.

We next show that the images of these maps lie in  the submodule $G^0T= T_0Q\widehat{\Pol}(A, \nabla,-1)$ of $T$. Given $a \in A$, we may rewrite the expression for $[D^{\nabla}_S,a]$ as the sum of iterated commutators $\sum_n (-\ad_a)(-\ad_S)^n(D^{\nabla})/n!$, and hence as $\sum_n (-\ad_a)(-\ad_S)^n(D^{\nabla}_{\ge n+1})/n!$, the lower order operators being annihilated. Since $S \in G^{-1}B$ and $D^{\nabla}_k \in G^{k-1}B$, this means that $[D^{\nabla}_S,a] \in G^0B$, and hence  $[D^{\nabla}_S-\delta,a] \in G^0B$. Thus $\mu'(F^1(A\ten A), S)\subset G^0B$;
 since $F^1(A\ten A)$ topologically generates  $\DR'(A/R)$ under multiplication, $\mu'(\DR'(A/R), S)$ is thus contained in  the subalgebra $G^0B$ of $B$, and the result follows by evaluation at $1$. The statements  for $\nu$ are an immediate consequence.

Finally, to see that $\mu(-,S)$ is a chain map, we may appeal to \cite[Lemma \ref{DQvanish-keylemma}]{DQvanish}, which gives the expression
\begin{align*}
D^{\nabla}_S\mu'(\omega, S) = \mu'( (d +\delta)\omega, S) + \nu'(\omega, S, \half (D^{\nabla}_S)^2),
\end{align*}
the final term vanishing because $D^{\nabla}_S$ is square-zero. That $\mu$ is a chain map then follows by evaluation at $1 \in T$.
\end{proof}
%

\begin{definition}
We say that a generalised  $(-2)$-shifted  pre-symplectic structure $\omega$ and an $E_{-1}$ quantisation $S$ of a 
 flat right connection $(A, \nabla)$
  are  compatible (or a compatible pair) if 
\[
 [\mu(\omega, S)] = [-\pd_{\hbar^{-1}}(S)] \in  \H^0((G*\tilde{F})^2T_{S}Q\widehat{\Pol}(A, \nabla,-2)),
\]
where $\sigma=-\pd_{\hbar^{-1}}$ is the canonical tangent vector of Definition \ref{Qsigmadef}. 
\end{definition}

\begin{definition}\label{vanishingdef}
Given a simplicial set $Z$, an abelian group object $A$ in simplicial sets over $Z$,  a space $X$ over $Z$ and a morphism  $s \co X \to A$ over $Z$, define the homotopy vanishing locus of $s$ over $Z$ to be the homotopy limit of the diagram
\[
\xymatrix@1{ X \ar@<0.5ex>[r]^-{s}  \ar@<-0.5ex>[r]_-{0} & A \ar[r] & Z}.
\]
\end{definition}

\begin{definition}\label{Qcompdef}
Define the space $Q\Comp(A, \nabla,-2)$ of compatible quantised $(-2)$-shifted pairs to be the homotopy vanishing locus of  
\[
 (\mu - \sigma) \co G\PreSp(A/R,-2) \by Q\cP(A,\nabla,-2) \to TQ^{tw}\cP(A,\nabla,-2)
\]
over $Q^{tw}\cP(A, \nabla,-2)$.  Note that there is no twist on the left, but that $\mu$ forces us to have twists on the right.

We define a cofiltration on this space by setting $ Q\Comp(A,\nabla,-2)/G^k$ to be the homotopy vanishing locus of  
\[
 (\mu - \sigma) \co (G\PreSp(A/R,-2)/\hbar^k)  \by (Q\cP(A,\nabla,-2)/G^k)  \to TQ^{tw}\cP(A,\nabla,-2)/G^k 
\]
over $Q^{tw}\cP(A,\nabla,-2)/G^k $.
\end{definition}

When $k=1$, note that this recovers the notion of compatible $(-2)$-shifted pairs from \cite{poisson}, because as in the proof of Proposition \ref{QPolmudef}, we have 
\begin{align*}
 \mu'(da,S)&= -\delta a + \sum_{n\ge 0}     [[\ldots [D^{\nabla}, \overbrace{S], \ldots ,S}^n],a]/n!\\
&\equiv   \sum_{n\ge 1}     [[\ldots [D^{\nabla}_{n+1}, S], \ldots ,S],a]/n! \mod G^1,\\
&= \sum_{n\ge 1}  [S,S, \ldots,S,a]_{\nabla,n+1}/n!
\end{align*}
which by Lemma \ref{DRrBVlemma} is just the Schouten--Nijenhuis bracket $[S,a]_{\nabla,2}$, the higher brackets vanishing.

\begin{definition}
 Define $Q\Comp(A,\nabla,-2)^{\nondeg} \subset Q\Comp(A,\nabla,-2)$ to consist of compatible quantised pairs $(\omega, \Delta)$ with $\Delta$ non-degenerate. This is a union of path-components, and by \cite[Lemma \ref{poisson-compatnondeg}]{poisson} has a natural map 
\[
 Q\Comp(A,\nabla,-2)^{\nondeg}\to G\Sp(A/R,-1)
\]
as well as the canonical map
\[
 Q\Comp(A,\nabla,-2)^{\nondeg} \to Q\cP(A,\nabla,-2)^{\nondeg}.
\]
\end{definition}

\begin{proposition}\label{QcompatP1} 
For any flat right connection $(A, \nabla)$, the canonical map
\begin{eqnarray*}
    Q\Comp(A,\nabla,-2)^{\nondeg} \to  Q\cP(A,\nabla,-2)^{\nondeg}           
\end{eqnarray*}
 is a weak equivalence. In particular, there is a morphism
\[
  Q\cP(A,\nabla,-2)^{\nondeg} \to G\Sp(A/R,-2)
\]
in the homotopy category of simplicial sets.
\end{proposition}
\begin{proof}
We adapt the proof of \cite[Proposition \ref{poisson-compatP1}]{poisson} and \cite[Proposition \ref{DQvanish-QcompatP1}]{DQvanish}.
For any $S \in Q\cP(A,\nabla,-2)$, the homotopy fibre of $Q\Comp(A/R,-2)^{\nondeg} $ over $S$ is just the homotopy fibre of
\[
\mu(-,S)  \co G\PreSp(A/R,-2)  \to T_{S}Q^{tw}\cP(A,\nabla,-2) 
\]
over $-\pd_{\hbar^{-1}}(S) \in Q^{tw}\cP(A,\nabla,-2)$.

The map $\mu(-,S) \co \DR'(A/R)\llbracket\hbar\rrbracket \to T_{S}Q\widehat{\Pol}(A,\nabla,-2)$ is a morphism of complete $(G*\tilde{F})$-filtered $R\llbracket\hbar\rrbracket$-modules. Since the morphism is $R\llbracket\hbar\rrbracket$-linear, it maps $G^k(G*\tilde{F})^p\DR'(A/R)\llbracket\hbar\rrbracket$ to $   G^k(G*\tilde{F})^pT_{S}Q\widehat{\Pol}(A,\nabla,-2)$. Non-degeneracy of $S_2$ modulo $F_1$ implies that $\mu(-,S)$ induces  quasi-isomorphisms 
\[
  \hbar^{k}\Omega^{p-2k}[2k-p] \to \hbar^{p-k}\HHom_A(\Omega^{p-2k}_{A/R}, A)[p-2k]
\]
on the associated gradeds $\gr_G^k\gr_{(G*\tilde{F})}^p$.  We therefore have a quasi-isomorphism of bifiltered complexes, so we have isomorphisms on homotopy groups:
\begin{eqnarray*}
 \pi_jG\PreSp(A/R,-2)  &\to& \pi_jT_{S}Q^{tw}\cP(A,\nabla,-2)\\
 \H^{-j}((G*\tilde{F})^2 \DR(A/R)\llbracket\hbar\rrbracket) &\to&  \H^{-j}((G*\tilde{F})^2T_{S}Q\widehat{\Pol}(A,\nabla,-2)).\qedhere
\end{eqnarray*}
\end{proof}

\subsection{Comparing quantisations and generalised symplectic structures}\label{comparisonsn}

\begin{definition}\label{Ndef}
Given a compatible pair  $(\omega, \pi) \in  \Comp(A,-2)= Q\Comp(A,\nabla,-2)/G^1$, and $k \ge 0$, define the complex 
$
 N(\omega,\pi,k) 
$
to be the cocone of the map 
\begin{align*}
 \gr_G^k(G*\tilde{F})^2(\DR'(A/R)\llbracket\hbar\rrbracket\oplus \gr_G^k\tilde{F}^2Q\widehat{\Pol}(A,\nabla,-2) 
 \to \gr_G^k(G*\tilde{F})^2T_{\pi}Q\widehat{\Pol}(A,\nabla,-2)
\end{align*}
given by combining                                
\begin{align*}
 \gr_G^k  \mu(-,\pi) \co  \gr_G^k(G*\tilde{F})^2\DR'(A/R)\llbracket\hbar\rrbracket &\to  \gr_G^k(G*\tilde{F})^2T_{\pi}Q\widehat{\Pol}(A,\nabla,-2) \\ 
\hbar^{k}F^{2-2k}\DR(A/R) &\to\prod_{i \ge (2-2k),0} \hbar^{i+k}\HHom_A(\Omega^i_{A/R},A)[i]
\end{align*}
with the maps 
\begin{align*}
 \gr_G^k\nu(\omega, \pi) + \pd_{\hbar^{-1}} \co (\gr_G^k\tilde{F}^2Q\widehat{\Pol}(A,\nabla,-2), \delta_{\pi}) &\to \gr_G^k(G*\tilde{F})^2T_{\pi}Q\widehat{\Pol}(A,\nabla,-2)\\
\prod_{i \ge (2-k),0}\hbar^{i+k-1} \HHom_A(\Omega^i_{A/R},A)[i] &\to  \prod_{i \ge (2-2k),0} \hbar^{i+k}\HHom_A(\Omega^i_{A/R},A)[i],
\end{align*}
 where                              
\[
 \nu(\omega, \pi)(b):= \nu(\omega, \pi, b).
\]
\end{definition}

It follows from the proof of  Proposition \ref{QcompatP1} that  the maps $\gr_G^k  \mu(-,\pi)$ are all $F$-filtered quasi-isomorphisms when $\pi$ is non-degenerate, so the projection maps $N(\omega,\pi,k) \to \gr_G^k\tilde{F}^2Q\widehat{\Pol}(A,\nabla,-2)$ are also quasi-isomorphisms. The behaviour of the other projection is more subtle for low $k$, but it behaves well thereafter, the proof of \cite[Lemma \ref{DQvanish-tangentlemma}]{DQvanish} adapting verbatim to give:

\begin{lemma}\label{tangentlemma}
The projection maps 
\[
 N(\omega,\pi,k) \to \hbar^{k}\DR(A/R)
\]
 are $F$-filtered quasi-isomorphisms for all $k \ge 2$.
\end{lemma}
%
%

\subsubsection{The comparison}

\begin{proposition}\label{quantprop}
The   maps
\begin{align*}
 Q\cP(A,\nabla,-2)^{\nondeg}/G^k &\to (Q\cP(A,\nabla,-2)^{\nondeg}/G^2)\by^h_{(G\Sp(A,-2)/G^2)}(G\Sp(A,-2)/G^k) \\ 
&\simeq (Q\cP(A,\nabla,-2)^{\nondeg}/G^2)\by \prod_{i=2}^{k-1} \mmc(\hbar^i\DR(A/R)[-1])
\end{align*}
coming from Proposition \ref{QcompatP1}  are weak equivalences for all $k \ge 2$.
\end{proposition}
\begin{proof}
Proposition \ref{QcompatP1} gives equivalences between $Q\cP^{\nondeg}$ and $Q\Comp^{\nondeg}$.  Fix $(\omega, \pi) \in \Comp(A,-2)$ and denote homotopy fibres by subscripts.
Arguing as in  the proof of \cite[Proposition \ref{DQvanish-quantprop}]{DQvanish}, Lemma \ref{tangentlemma} shows that for $k \ge 2$, the right-hand map is a weak equivalence in the commutative diagram
\[
\begin{CD}
 (Q\Comp(A,\nabla,-2)/G^{k+1})_{(\omega, \pi)} @>>> (Q\Comp(A,\nabla,-2)/G^k)_{(\omega,\pi)} @>>> \mmc(N(\omega,\pi,k))\\
@VVV @VVV @VVV \\
(G\PreSp(A,-2)/G^{k+1})_{\omega}@>>> (G\PreSp(A,-2)/G^{k})_{\omega} @>>> \mmc(\hbar^{k}F^{2-2k}\DR(A/R))
\end{CD}
\]
of homotopy  fibre sequences, so
\[
 (Q\Comp(A,\nabla,-2)/G^k)\by^h_{G\PreSp(A,-2)/G^{k} }G\PreSp(A,-2)/G^{k+1},
\]
and the result follows by induction.
\end{proof}

\begin{remark}\label{quantrmk}
 Taking the limit over all $k$, Proposition \ref{quantprop}  gives an equivalence
\[
  Q\cP(A,\nabla,-2)^{\nondeg} \simeq (Q\cP(A,\nabla,-2)^{\nondeg}/G^2)\by \prod_{i \ge 2} \mmc(\hbar^i\DR(A/R)[-1]);
\]
in particular, this means that there is a canonical map 
\[
 (Q\cP(A,\nabla,-2)^{\nondeg}/G^2) \to Q\cP(A,\nabla,-2)^{\nondeg},
\]
corresponding to the distinguished point $0 \in \mmc( \hbar^2\DR(A/R)[-2]\llbracket\hbar\rrbracket)$.

Thus to quantise a non-degenerate $(-2)$-shifted Poisson structure $\pi_{\hbar} =\sum_{j \ge 2} \hbar^{j-1}\pi_j$ (or equivalently, by \cite[Corollary \ref{poisson-compatcor2}]{poisson}, a $(-2)$-shifted symplectic  structure), it suffices to lift the power series $\pi_{\hbar}$ to a Maurer--Cartan element of the $L_{\infty}$-algebra $\prod_{j \ge 2} \hbar^{j-1}(F_j\DR^r(A,\nabla)/F_{j-2})$.

Even in the degenerate case, the proof of Proposition \ref{quantprop} gives a sufficient first-order criterion  for quantisations to exist:
\[
 Q\Comp(A,\nabla,-2) \simeq (Q\Comp(A,\nabla,-2)/G^2)\by \prod_{i \ge 2} \mmc(\hbar^i\DR(A/R)[-1]).
\]
\end{remark}

For the $(-1)$-shifted and $0$-shifted quantisations considered in \cite{DQvanish, DQnonneg}, there was a notion of self-duality for quantisations, and restricting to these allowed us to show that even the first-order obstruction to quantising non-degenerate Poisson structures vanishes. The following lemma and example show that the same is not true for our notion of   $(-2)$-shifted quantisations, and that we need an additional condition on compatibility of the Poisson structure and the connection.

\begin{lemma}\label{liftobslemma} 
 Given a $(-2)$-shifted Poisson structure $\pi_{\hbar} =\sum_{j \ge 2} \hbar^{j-1}\pi_j$, the obstruction $\ob_{\nabla}(\pi_{\hbar})$ to lifting $\pi_{\hbar}$ to $Q\cP(A,\nabla,-2)/G^2$ (resp. $Q^{tw}\cP(A,\nabla,-2)/G^2$) 
is given by the class
\[
[\sum_{n \ge 1} [\pi_{\hbar}, \ldots, \pi_{\hbar}]_{\nabla_{n+1},n}/n!] 
\]
in $\H^1(F^1T_{\pi} \widehat{\Pol}(A/R,-2))$ (resp. $\H^1(T_{\pi} \widehat{\Pol}(A/R,-2))$).
\end{lemma}
\begin{proof}
 Just observe that $\pi_{\hbar}$ naturally defines an element of 
\[
 (G*\tilde{F})^2Q\widehat{\Pol}(A/R,-2)/G^2= F_0\DR^r(A,\nabla) \by \prod_{j\ge 0}\hbar^{j-1}\DR^r(A,\nabla)/F_{j-2}),
\]
and that the Maurer-Cartan expression is given by
\[
 \sum_{n \ge 1} [\pi_{\hbar},\ldots,\pi_{\hbar} ]_{\nabla,n}/n!= \sum_{n \ge 1} [\pi_{\hbar}, \ldots,\pi_{\hbar}]_{\nabla_{\ge n},n}/n!.
\]
Since $\pi$ is a Poisson structure, we know that 
\[
 \sum_{n \ge 1} [\pi_{\hbar},\ldots,\pi_{\hbar}]_{\nabla_{ n},n}/n!=0,
\]
and we also know that the terms $[\pi_{\hbar},\ldots,\pi_{\hbar}]_{\nabla_{ \ge n+2},n}$ lie in $G^2Q\widehat{\Pol}$, leaving only the terms $[\pi_{\hbar},\ldots,\pi_{\hbar}]_{\nabla_{n+1},n}$ to contribute to the obstruction.
\end{proof}

\begin{example}
Consider the shifted cotangent bundle $T^*\bG_m[2]$, whose ring of functions is given by $A:=R[x,x^{-1},\xi]$ for $x$ of degree $0$ and $\xi$ of cochain degree $-2$. Consider the $(-2)$-shifted symplectic structure $xdx d\xi$, with associated non-degenerate $(-2)$-shifted Poisson structure $\pi_{\hbar}:=\hbar x^{-1}\pd_{\xi} \pd_x$. Now consider the right $\sD$-module structure on $A$ given by the connection $ \nabla_2 (a\pd_x + b\pd_{\xi}) =\frac{\pd a}{\pd x} + \frac{\pd b}{\pd \xi}$ (induced by the isomorphism $f \mapsto f dxd\xi$ from $A$ to the natural right $\sD$-module $\Omega^2_A[-2]$).

The obstruction of Lemma \ref{liftobslemma} is given by $[D^{\nabla}_2(\pi_{\hbar})]= [\hbar x^{-2}\pd_{\xi}] \in  \H^1(T_{\pi} \widehat{\Pol}(A/R,-2))$. Under the isomorphism $\mu(-,\pi_{\hbar})$, this corresponds to the element $[x^{-1}dx] \in \H^1(\DR(A/R)\cong\H^1(\bG_m/R)$, which is non-zero, so in this case the homotopy fibres of  $Q\cP(A,\nabla,-2)/G^2\to \cP(A,-2)$ over $\pi_{\hbar}$ are empty.

\end{example}

\begin{corollary}\label{liftobscor} 
 If the scheme $\Spec \H^0A$ is connected and  $R=\H^0R$, then whenever the obstruction of Lemma \ref{liftobslemma} vanishes, the twisted quantisation of a given non-degenerate $(-2)$-shifted Poisson structure is essentially unique up to addition by $\hbar R\brh$.  
\end{corollary} 
\begin{proof}
Since scalars $R$ are untouched by all the operations, the additive group $\hbar R\brh$ acts on the space $Q^{tw}\cP(A,\nabla,-2)$ by addition:  if $S$ is a  twisted quantisation then so is $S+ r(\hbar)$ for $r(\hbar) \in \hbar R\brh$. Also note that the pair $(\omega+ \hbar^2r'(\hbar), S + r(\hbar))$ are compatible whenever $(\omega, S) \in Q\Comp(A,\nabla,-2)$.

The hypotheses imply that $\H^0\DR(A/R)=R$, with $\H^{<0}\DR(A/R)=0$, and hence also  $\H^{\le 0}T_{\pi} \widehat{\Pol}(A/R,-2)\cong R$ for  non-degenerate Poisson structures $\pi$, via the isomorphism $\mu(-,\pi)$. When the obstruction of  Lemma \ref{liftobslemma} vanishes, Proposition \ref{quantprop} ensures that the space of twisted quantisations of $\pi$ is non-empty, and comparison of tangent spaces for the tower $\{Q^{tw}\cP(A,\nabla,-2)/G^k\}_k$ shows that $Q^{tw}\cP(A,\nabla,-2)_{\pi}$ must be an  $\hbar R\brh$-torsor under the action above.
\end{proof}

\subsubsection{Compatibility of connections and symplectic structures}

We now show that varying the right $\sD$-module structure allows us to eliminate the obstruction of Lemma \ref{liftobslemma}, and that there is a unique choice which does so.

\begin{lemma}\label{alphanablalemma}
Given $\alpha = \sum_{p \ge 1} \alpha_p \in \z^1(F^1\DR(A))$ and a flat right connection $\nabla$ on $A$ in the sense of Definition \ref{htpyrightDmoddef}, there is a right $\sD$-module structure $\nabla^{\alpha}$ given by 
\[
 \nabla_{p+1}^{\alpha}= \nabla_{p+1} + \lrcorner\alpha_p  \co \HHom_A(\Omega^{p}_{A/R},A)^{\#} \to A^{\#}[1-p].
\]
\end{lemma}
\begin{proof}
The only non-trivial condition to check is that $\{ \nabla_{p+1}^{\alpha}\}_{p \ge 1}$ defines an $L_{\infty}$-derivation for the opposite module structure, or equivalently that $D^{\nabla^{\alpha}}\circ D^{\nabla^{\alpha}}=0$. Now, observe that 
\[
 D^{\nabla^{\alpha}}(\pi)= D^{\nabla}(\pi) + \pi\lrcorner \alpha,
\]
so 
\begin{align*}
 D^{\nabla^{\alpha}}\circ D^{\nabla^{\alpha}}&= D^{\nabla}\circ D^{\nabla} + (\lrcorner\alpha )\circ  D^{\nabla}+  D^{\nabla}\circ (\lrcorner\alpha ) + \lrcorner(\alpha\wedge \alpha)\\
&=(\lrcorner\alpha )\circ  D^{\nabla}+  D^{\nabla}\circ (\lrcorner\alpha).
\end{align*}

For $\pi \in\HHom_A(\Omega^p_{A},A)[p]$ and $\omega \in \Omega^q_A$, we have $D^{\nabla}(\pi)(\omega)= \nabla(\pi \lrcorner \omega) \pm \pi(d\omega)+ (\delta \pi)(\omega)$, so 
\begin{align*}
 ((\lrcorner\alpha )\circ  D^{\nabla}(\pi))(\omega)&=  \mp D^{\nabla}(\pi)(\alpha \wedge \omega)\\
&=\mp\nabla(\pi \lrcorner (\alpha \wedge \omega)) \mp \pi(d(\alpha \wedge\omega))\mp 
\delta( \pi(\alpha \wedge\omega)) \pm \pi(\delta(\alpha \wedge\omega)).\\
(D^{\nabla}\circ (\lrcorner\alpha ))(\pi)(\omega)&= \pm D^{\nabla}(\pi \lrcorner\alpha )(\omega)\\
&=\pm\nabla(\pi \lrcorner (\alpha \wedge \omega))\pm \pi(\alpha \wedge d\omega) \pm 
\delta( \pi(\alpha \wedge\omega)) \mp \pi(\alpha \wedge\delta\omega).
\end{align*}
Cancelling terms, this gives
\[
 D^{\nabla^{\alpha}}( D^{\nabla^{\alpha}}\pi)(\omega) = \mp\pi(d\alpha \wedge \omega) \pm \pi(\delta\alpha \wedge \omega),
\]
but $d\alpha \mp \delta \alpha=0$ because $\alpha \in \z^1(F^1\DR(A))$.
\end{proof}

\begin{lemma}\label{obalphalemma}
  Given a $(-2)$-shifted Poisson structure $\pi_{\hbar} =\sum_{j \ge 2} \hbar^{j-1}\pi_j$, a flat right connection $\nabla$ on $A$, and an element $\alpha = \sum_{p \ge 1} \alpha_p \in \z^1(F^1\DR(A))$,
the difference
\[
 \ob_{\nabla^{\alpha}}(\pi_{\hbar})-\ob_{\nabla}(\pi_{\hbar}) \in \H^1(F^1T_{\pi} \widehat{\Pol}(A/R,-2))
\]
between the obstructions to lifting $\pi_{\hbar}$ to $Q\cP(A,\nabla^{\alpha},-2)/G^2$ or to $Q\cP(A,\nabla,-2)/G^2$ (cf. Lemma \ref{liftobslemma}) is 
\[
 \mu(\alpha, \pi_{\hbar}),
\]
for the compatibility map $\mu(-,\pi_{\hbar}) \co\DR(A) \to T_{\pi_{\hbar}}\cP(A,-2)$ of \cite[Definition \ref{poisson-mudef}]{poisson}, a multiplicative map given on generators by
$\mu(a df, \pi_{\hbar}):= \pi_{\hbar} \lrcorner (a df)$.
\end{lemma}
\begin{proof}
 Since $\ob_{\nabla}(\pi_{\hbar})= \sum_{n \ge 1} [\pi_{\hbar}, \ldots, \pi_{\hbar}]_{\nabla_{n+1},n}/n!$, we have
\[
 \ob_{\nabla^{\alpha}}(\pi_{\hbar})-\ob_{\nabla}(\pi_{\hbar}) = \sum_{n \ge 1} [\pi_{\hbar}, \ldots, \pi_{\hbar}]_{\lrcorner \alpha_n,n}/n!,
\]
where $[v_1, \ldots, v_n]_{\lrcorner \alpha_n,n}= [\ldots [\lrcorner \alpha_n, v_1], \ldots , v_n](1)$. 

Now, for the insertion operator $i$, we have $[\lrcorner \alpha, v]= \lrcorner i_v(\alpha_n)$, so 
\[
 [\pi_{\hbar}, \ldots, \pi_{\hbar}]_{\lrcorner \alpha_n,n}/n!= 1 \lrcorner (i_{\pi_{\hbar}}^{\circ n}(\alpha_n))/n!,
\]
which is just $\mu(\alpha_n, \pi_{\hbar})$.
\end{proof}

\begin{definition}
 We define the space of  flat right connections on $A$ over $R$ to be the simplicial set given in level $n$ by the set of flat right connections on $A\ten_{\Q} \Omega^{\bt}(\Delta^n)$ over $R\ten_{\Q} \Omega^{\bt}(\Delta^n)$, with the obvious simplicial operations.
\end{definition}

\begin{proposition}\label{uniqueconn}
Given a non-degenerate $(-2)$-shifted Poisson structure $\pi_{\hbar} =\sum_{j \ge 2} \hbar^{j-1}\pi_j$ on $A$ over $R$, there is essentially at most one pair $(\nabla, S)$ where $\nabla$ is a flat right connection on $A$ and $S$ is a first-order deformation quantisation of $\pi_{\hbar}$ relative to $\nabla$.

Explicitly, the space of pairs $(\nabla, S)$, for $S$ in the homotopy fibre of   $Q\cP(A,\nabla,-2)/G^2 \to \cP(A,-2)$ over $\pi_{\hbar}$, is either empty or contractible, depending on whether any flat right connections on $A$ exist. 
\end{proposition}
\begin{proof}
If there do not exist  flat right connections on $A$, then the space is empty. Otherwise, choose a connection $\nabla^0$. Lemma \ref{alphanablalemma} gives a morphism $\alpha \mapsto \nabla^{0,\alpha}$ from $\mmc(F^1\DR(A))$ to the space of flat right connections. It follows from the non-degeneracy hypothesis that  each $\Omega^p_A$ is perfect as an $A$-module, so the map from $\Omega^p_A$ to its double dual is a quasi-isomorphism. Obstruction calculus as in \cite[\S \ref{poisson-towersn}]{poisson} then shows that $\mmc(F^1\DR(A))$ is weakly equivalent to the space of flat right connections.

By Lemma \ref{obalphalemma}, the space of  pairs $(\nabla, S)$ as above for varying $\pi$ is given by the homotopy fibre of
\begin{align*}
 \mmc(F^1\DR(A)) \by \cP(A,-2)^{\nondeg} &\to  \mmc(F^1T_{\pi} \widehat{\Pol}(A/R,-2))\\
(\alpha, \pi) \mapsto \ob_{\nabla^0}(\pi) +\mu(\alpha, \pi)
\end{align*}
over $0$. Since $\pi$ is non-degenerate, the map $\mu(-,\pi)$ is a filtered quasi-isomorphism, so the natural map  from this space of pairs down to $ \cP(A,-2)^{\nondeg}$ is a weak equivalence. In other words, 
$\pi$ admits an essentially unique first-order quantisation compatible with an essentially  unique flat right connection.
\end{proof}


Combining Proposition \ref{uniqueconn} with Proposition \ref{quantprop} and the proof of Corollary \ref{liftobscor}, we have:
\begin{corollary}\label{uniqueconncor}
Take a non-degenerate $(-2)$-shifted Poisson structure $\pi_{\hbar} =\sum_{j \ge 2} \hbar^{j-1}\pi_j$ on $A$ over $R$, with $\Spec \H^0A$ connected and  $R=\H^0R$. If $A$ admits any flat right connections, then   pairs $(\nabla, S)$, with $\nabla$  a flat right connection on $A$ and $S \in Q\cP(A, \nabla,-2)$ a  quantisation  of $\pi_{\hbar}$, are essentially unique up to addition by $(0,\hbar^2 R\brh)$.
\end{corollary}

\section{Global quantisations}\label{globalsn}

The derived affine schemes considered  in \S \ref{affinesn} are a fairly limited class of objects with which to work, so in this section we indicate how to formulate and study quantisations of $(-2)$-shifted symplectic structures on derived Deligne--Mumford stacks or on derived Artin stacks. The generalisation proceeds along much the same lines as \cite[\S\S \ref{DMsn}, \ref{Artinsn}]{DQvanish} and \cite[\S\S \ref{DMsn}, \ref{Artinsn}]{poisson}, so we concentrate on those features which are specific to our setting.

\subsection{\'Etale functoriality and derived Deligne--Mumford stacks}\label{DMsn}

In \S \ref{cfBJ}, we saw some $(-2)$-shifted quantisations on non-affine schemes, defined in terms of strict functoriality, but the definitions and constructions of \S \ref{affinesn} adapt much more generally using derived \'etale functoriality.

The construction of \cite[Definitions \ref{poisson-inftyFXdef} and Definition \ref{poisson-DMFdef}]{poisson} uses homotopy \'etale descent to construct, for any strongly quasi-compact derived Deligne--Mumford $n$-stack $\fX$ over $R$,  a simplicial set $F(\fX)$ associated to any $\infty$-functor $F$ on homotopy \'etale morphisms of $R$-CDGAs. As in  \cite[\S \ref{poisson-descentsn}]{poisson}, such  $\infty$-functors can be constructed from any construction $F$ on fibrant cofibrant $[m]$-diagrams  of  $R$-CDGAs satisfying 
\cite[Properties \ref{poisson-Fproperties}]{poisson}.
%
%
%
%
In order to construct $Q\widehat{\Pol}(\fX, \nabla,-2)$, we need such a construction for right connections $\nabla $ and for right de Rham complexes. 

By \cite[Lemma \ref{poisson-calcTlemma2}]{poisson}, a construction exists for the tangent sheaf $\oR\HHom_{\sO_{\fX}}(\Omega^1_{\fX},\sO_X)$ satisfying
\cite[Properties \ref{poisson-Fproperties}]{poisson}, while \cite[Definition \ref{DQvanish-DMDdef}]{DQvanish} extends this construction to the DGAA $\sD_{\fX/R}$ of differential operators. In particular, we have an $\infty$-functor $A \mapsto \oR F_1\sD_{A/R}$ of dg Atiyah algebras on the site of homotopy \'etale affines over $\fX$. 

Moreover, the left $\sD$-module structure of $\sO_{\fX}$
 induces a functorial decomposition $\oR F_1\sD_{A/R} \cong \oR F_0\sD_{A/R} \oplus \oR \gr^F_1\sD_{A/R}$ of left $\oR F_0\sD_{A/R}$-modules which respects the commutator Lie bracket (of weight $-1$). 
In particular, the decomposition makes $\oR \gr^F_1\sD_{A/R}$ a dg Lie--Rinehart algebra (or Lie algebroid) over $\oR F_0\sD_{A/R}$, and these are resolutions of the tangent sheaf and structure sheaf respectively.

Constructing a flat right connection $\nabla$ on $\sO_{\fX}$ then amounts to constructing a functorial  homotopy anti-involution on the Atiyah algebra $\oR F_1\sD_{-/R}$, or equivalently on its universal enveloping algebra $\oR \sD_{-/R}$. Applying the description of  \cite[Remarks \ref{DQpoisson-noninvolutivermk} and \ref{DQpoisson-filteredoperadrmk2}]{DQpoisson} functorially, the potential obstruction to such a right connection existing lies in $\H^2(F^1\DR(\fX/R))$, while if the obstruction vanishes the space of flat right connections is a torsor for the additive group space $\mmc(F^1\DR(\fX/R))$. 

Adapting the relevant Definitions along the lines of \cite[Definitions \ref{poisson-inftyFXdef} and Definition \ref{poisson-DMFdef}]{poisson}, Propositions \ref{quantprop} and \ref{uniqueconn}
 then adapt with the results above to give:

\begin{proposition}\label{DMquantprop}
Given a strongly quasi-compact derived Deligne--Mumford $n$-stack $\fX$ locally of finite type over $R$, and a homotopy right connection $\nabla$ on $\sO_{\fX}$,
the   maps
\begin{align*}
 Q\cP(\fX,\nabla,-2)^{\nondeg}/G^k &\to (Q\cP(\fX,\nabla,-2)^{\nondeg}/G^2)\by^h_{(G\Sp(\fX,-2)/G^2)}(G\Sp(\fX,-2)/G^k) \\ 
&\simeq (Q\cP(\fX,\nabla,-2)^{\nondeg}/G^2)\by \prod_{i=2}^{k-1} \mmc(\hbar^i\DR(\fX/R)[-1])
\end{align*}
coming from Proposition \ref{QcompatP1}  are weak equivalences for all $k \ge 2$.

Moreover, for any non-degenerate $(-2)$-shifted Poisson structure $\pi_{\hbar}$ on $\fX$,
 the space of pairs $(\nabla, S)$, for $S$ in the homotopy fibre of   $Q\cP(\fX,\nabla,-2)/G^2 \to \cP(\fX,-2)$ over $\pi_{\hbar}$, is either empty or contractible, depending on whether any flat right connections exist on $\sO_{\fX}$; the potential  obstruction lies in $\H^2(F^1\DR(\fX/R))$.
\end{proposition}

%
%
%

\subsection{Derived Artin stacks}\label{Artinsn}

The stacky CDGAs (commutative bidifferential  bigraded algebras) of \cite[\S \ref{poisson-stackyCDGAsn}]{poisson} model formal completions of affine atlases over derived Artin $n$-stacks, giving  formally \'etale resolutions of derived Artin $n$-stacks by affine objects. Polyvectors and differential operators satisfy formally \'etale functoriality as in \cite[\S \ref{Artinsn}]{poisson} and \cite[\S \ref{Artinsn}]{DQvanish}, so the reasoning of \S \ref{DMsn} adapts, with Proposition \ref{DMquantprop} adapting verbatim for derived Artin $n$-stacks. 

Note that the right de Rham complexes involved in formulating $(-2)$-shifted quantisations for derived Artin $n$-stacks are thus defined in terms of stacky CDGAs, giving rise to complexes of quantised $(-2)$-shifted polyvectors which are formal deformations of the complexes of polyvectors from  \cite[\S \ref{Artinsn}]{poisson}.

\section{Virtual fundamental classes from quantisations }\label{cfBJ}

\subsection{Right de Rham cohomology and Borel--Moore homology}

\subsubsection{Borel--Moore homology as cohomology of $\omega_X$}

On a derived $R$-scheme $X$ which is locally of finite presentation,  
\cite[Corollary 5.6.8]{GaitsgoryIndCoh} constructs  an ind-coherent dualising complex $\omega_X$ on $X$. We will use the same notation for the associated quasi-coherent complex, which suffices for our limited pruposes.
For an affine dg manifold $\oR \Spec A$, the dualising complex is given   by $\omega_A= \HHom_{A_0}(A, \Omega^n_{A_0})[n]$ when $A_0$ has dimension $n$. Since $\omega_X$ is a right $\sD$-module in the sense of \cite{GaitsgoryRozenblyumCrystal}, by \cite[Example \ref{DQvanish-stratDAex}]{DQvanish} it is a right $\sD$-module in our sense. 

The following generalises Definition \ref{DRrdef}:
\begin{definition}\label{DRrdef2}
 Given a morphism $f \co X \to S$ of derived schemes and a right $\sD_{X/S}$-module $\sE$ on $X$,  define the hypersheaf $\oL\DR^r_{X/S}(\sE)$ on $X$ by
\[
 \oL\DR^r_{X/S}(\sE):= \sE\ten^{\oL}_{\sD_{X/S}}\sO_X.
\]
\end{definition}

\begin{lemma}\label{BMlemma} 
For  a quasi-compact quasi-separated derived scheme $X$ locally of finite presentation over $\Cx$, with underived truncation $\pi^0X$, we have
\[
 \H_d^{BM}(\pi^0X(\Cx),\Cx)\simeq  \bH^{-d}(X, \oL\DR^r(\omega_X)). 
\]
\end{lemma}
\begin{proof}[Proof (sketch)]
We use the characterisation of Borel--Moore homology $\H_d^{BM}(\pi^0X(\Cx),\Cx)$ as cohomology $\H^{-d}(\pi^0X(\Cx)_{\an},\bD_{\pi^0X(\Cx)})$ of the $\Cx$-dualising complex $\bD_{\pi^0X(\Cx)}$ on the analytic site of $\pi^0X$. 
 Assume for simplicity that $X$ admits a derived closed immersion $i \co X \to X^0$ with $X^0$ smooth. For the associated closed immersion $\bar{i} \co \pi^0X \to X^0$, it suffices to show that the complexes $\oR \bar{i}_*\bD_{\pi^0X(\Cx)} \simeq    \oR \bar{i}_*\bar{i}^!\bD_{X^0(\Cx)} $ (on the analytic site of $X^0$) and $\oR i_* \DR^r(\omega_X) \simeq \oR i_* (\omega_X \ten^{\oL}_{\sD_X}\sO_X)$ (on the Zariski site of $X^0$)  have isomorphic hypercohomology.

 By \cite[6.1.2]{saitoDmodsAnalyticSpaces}, the  complex $\oR \bar{i}_*\bar{i}^!\bD_{X^0(\Cx)} $ is given by the right de Rham complex
\[
(\oR\Gamma_{[\pi^0X(\Cx)]}\omega_{X^0}^{\an})\ten_{\sD_{X^0}^{\an}}^{\oL}\sO_{X^0},
\]
where $ \Gamma_{[\pi^0X(\Cx)]}M = \LLim\hom_{\sO_{X^0}^{\an}}( \sO_{X^0}^{\an}/\sI_{\pi^0X}^{\an} ,M)$. Meanwhile, we can write
\[
 \omega_X \ten^{\oL}_{\sD_X}\sO_X \simeq \omega_X \ten^{\oL}_{\sD_X}\cDiff(i^{-1}\sO_{X^0}, \sO_X) \ten^{\oL}_{i^{-1}\sD_{X^0}}i^{-1}\sO_{X^0},
\]
and $\omega_X \ten^{\oL}_{\sD_X}\cDiff(i^{-1}\sO_{X^0}, \sO_X) \simeq  \omega_X \ten^{\oL}_{\sD_{X/X^0}}\sO_X$, where $\cDiff$ denotes the sheaf of differential operators.

It thus suffices to show that the complex $\oR\Gamma_{[\pi^0X(\Cx)]}\omega_{X^0}^{\an} $ of analytic right $\sD$-modules on $X^0$ is the analytification of the complex  $i_* (\omega_X \ten^{\oL}_{\sD_{X/X^0}}\sO_X)$ of algebraic $\sD$-modules.

Now, since $\omega_X \simeq \oR\hom_{i^{-1}\sO_{X^0}}(\sO_X, i^{-1}\omega_{X^0})$, we have 
\[
 \oR\hom_{\O_{X^0}}( i_* (\omega_X \ten^{\oL}_{\sD_{X/X^0}}\sO_X), \sO_{X^0})\simeq  i_*\oR\hom_{\sD_{X/X^0}}(  \sO_{X}\ten_{i^{-1}\sO_{X^0}} i^{-1}\omega_{X^0}, \sO_X).
\]
Since $i$ is  a derived closed immersion, the left de Rham complex $\DR(\sO_X/\sO_{X^0})$ is quasi-isomorphic to  the completion of $\sO_X$ along $\pi^0X$,
so the last expression above is just the completion of $ \omega_{X^0}^*$ with respect to the ideal $ \sI_{\pi^0X}$. Meanwhile, the $\sO_{X^0}^{\an}$-linear dual of $\oR\Gamma_{[\pi^0X(\Cx)]}\omega_{X^0}^{\an} $ is the completion of $ (\omega_{X^0}^{\an})^*$ with respect to the analytic ideal $\sI_{\pi^0X}^{\an}$, permitting the desired comparison along the lines of \cite{HartshorneAlgDeRham}.
\end{proof}

\begin{remark}
More generally, to right any right $\sD$-module we may associate a complex  $\oL\DR^r(\sE)^{\an}$ of sheaves on $\pi^0X(\Cx)^{\an}$. This respects the six functor formalism via the comparison in \cite[Example \ref{DQvanish-stratDAex}]{DQvanish} between our ind-coherent right $\sD$-modules and those of \cite{GaitsgoryRozenblyumCrystal}.
\end{remark}

\subsubsection{Reduction to Gorenstein derived schemes}\label{Gorensteinsn}

As in \S \ref{QME},
since the space $Q\cP(A, \nabla,-2)$ of $E_{-1}$ quantisations from Definition \ref{Qpoissdef} consists of solutions of the quantum master equation, any quantisation $S \in Q\cP(A, \nabla,-2)$ gives rise to a $0$-cocycle $e^S$ in the right de Rham complex $\DR^r(A, \nabla) \brh$. 

 If we write $\dim A$ for the virtual dimension of $\oR\Spec A$ over $R$, then a morphism $(A, \nabla)\to \omega_A[-\dim A]$ of right $\sD$-modules would then give us a class of degree $\dim A$ in Borel--Moore homology associated to $e^S$. However, $(-2)$-shifted symplectic derived schemes are seldom Gorenstein, so $\omega_A$ will not be a line bundle in the cases which interest us. Instead, we now establish some fairly general circumstances in which  the right de Rham complex $\DR^r(A,\nabla)$ is quasi-isomorphic to a shift of $\DR^r(\omega_A)$. Beware that for the derived schemes we consider, the structure sheaf is unbounded, so  not ind-coherent, ruling out direct comparisons with the right crystals of \cite{GaitsgoryRozenblyumCrystal}. 


\begin{example}
Consider the shifted cotangent space $T^*[-2]\bA^1$ of the affine line, corresponding to the CDGA $A:=R[x, \xi]$ with $x\in A_0$, $\xi\in A_2$. This has a natural $(-2)$-shifted symplectic structure $dx\wedge d\xi$, and corresponding  non-degenerate $(-2)$-shifted Poisson structure $\pi$ determined by the equation $\{x,\xi\}_{\pi}=1$. The essentially unique right connection $\nabla$ compatible with $\pi$ is given by $\nabla(a \pd_x + b\pd_{\xi})= -\frac{\pd a}{\pd x}- \frac{\pd b}{\pd \xi}$, and then contraction with $dx \wedge d\xi$ defines an isomorphism $\alpha \co \DR^r(A, \nabla)\to \DR(A)$, which in turn is quasi-isomorphic to the base ring $R$.
 
On the other hand, $\omega_A \cong ( R[x,\xi,\xi^{-1}]/R[x,\xi])dx \wedge d\xi[1]$, 
and then $\DR^r(\omega_A) \cong (\DR(A)[\xi^{-1}]/\DR(A))[1]$, 
which is quasi-isomorphic to $R[2]$ via the element $\xi^{-1}d\xi$, 
 so we do have $\DR^r(A, \nabla) \simeq \DR^r(\omega_A)[-2]$, which tallies well with $T^*[-2]\bA^1$ having virtual dimension $2$.

 However, comparison of the Hodge filtrations shows that the homotopy right $\sD$-module morphism $A \to\omega_A[-2]$ giving rise to this quasi-isomorphism is not a quasi-isomorphism, and is  in fact zero on the underlying $A$-modules, since the generator $\pd_x\pd_{\xi}$ of  $\DR^r(A, \nabla)$ lies in $F_2 \setminus F_1$, while the generator $ (\xi^{-1}d\xi\wedge dx) \pd_x$  of $\DR^r(\omega_A)$ lies in $F_1\setminus F_0$. 
 

Under the isomorphism $\DR^r(A, \nabla) \simeq \DR^r(\omega_A)[-2]$, the quantisation corresponding to the constant power series $dx\wedge d\xi$ is then simply $\hbar \pi =\hbar \pd_x\pd_{\xi}$, giving virtual fundamental class $\exp(\hbar \pd_x\pd_{\xi})= 1+\hbar \pd_x\pd_{\xi}\sim\hbar \pd_x\pd_{\xi} \in \H^0\DR^r(A, \nabla)\brh$. This corresponds to the class $\hbar \in \H^0\DR(A)\brh$ under the isomorphism $\alpha$ above, and hence to 
\[
\hbar[\bA^1] \in \H^{BM}_2(\bA^1, \Cx\brh)
\]
when $R=\Cx$, via Lemma \ref{BMlemma}. In general for $X$ smooth over $\Cx$, a similar argument (cf. Proposition \ref{kappaprop} below)
 gives the quantisation class  of a shifted cotangent complex as
\[
[T^*[-2]X]= \hbar^{\dim X}[X] \in \H^{BM}_{2\dim X}(X, \Cx\brh).
\]
%
\end{example}

The reduction in this example  of the right de Rham cohomology of a derived scheme to the right de Rham cohomology of a line bundle on a Gorenstein scheme generalises to the following lemma. 
In the lemma, we will be transforming $\sD$-modules using the $\sD_A-\sD_B$-bimodule $\cDiff(B,A)$ of differential operators from $B$ to $A$, which is isomorphic to $A\ten_B\sD_B$ whenever $A$ and $B$ are affine dg manifolds. 


\begin{lemma}\label{DRrGorensteinLemma} 
 Take a morphism $f \co X \to Y$ of dg $R$-manifolds   in the sense of \cite{Quot}, such that $\sO_X$ is locally freely generated as a graded algebra over $f^{-1}\sO_Y$ by $s$ local generators   in homological degree $2$.  

 Then  the construction $\DR^r_{X/Y}(-)$ defines a functor from homotopy  right $\sD_X$-modules $(\sF, \nabla)$ to homotopy  right $\sD_Y$-modules, and there is a natural quasi-isomorphism
 \[
\DR^r_{X/Y}(\sF,\nabla)\ten_{\sO_Y}\sO_X \to \hom_{\sO_X}(\Omega^s_{X/Y},\sF)[s]  
 \]
of $\sO_X$-modules.
 \end{lemma}
 \begin{proof}
 For a homotopy right $\sD_X$-module $(\sF,\nabla)$,  we have
\[
\DR^r_{X/Y}(\sF,\nabla)= \bigoplus_{p}\hom_{\sO_X}(\Omega^p_{X/Y},\sF)[p]
\]
with differential induced by  $D^{\nabla}$; this is a natural subcomplex of $\DR^r_{X/R}(\sF,\nabla) $. 
We can rewrite this as $\sF\ten_{\nabla,\sD_X}^{\oL}\cDiff(\sO_Y,\sO_X)$, and then the right $\sD_Y$-module structure on $\cDiff(\sO_Y,\sO_X)$ induces a homotopy right $\sD_Y$-module structure on $ \DR^r_{X/Y}(\sF,\nabla)$.

The hypotheses ensure that $\Omega^p_{X/Y}= 0$ for $p>s$, giving a natural $\sO_Y$-linear map
\[
 \DR^r_{X/Y}(\sF,\nabla) \to \hom_{\sO_X}(\Omega^s_{X/Y},\sF)[s].
\]
This in turn induces a map $\theta \co  \DR^r_{X/Y}(\sF,\nabla)\ten_{\sO_Y}\sO_X \to   \hom_{\sO_X}(\Omega^s_{X/Y},\sF)[s] 
$, and it remains to show that this is a quasi-isomorphism.

It follows immediately from Definition \ref{DRrdef} that the contraction map
\[
 \DR(\sO_X/\sO_Y) \ten_{\sO_Y} \DR^r_{X/Y}(\sF, \nabla) \xra{\lrcorner} \DR^r_{X/Y}(\sF, \nabla)
\]
is a chain map, and we can then characterise  $\HHom_{\sO_X}(\det \Omega^1_{X/Y},\sF)[-s]$ as the quotient of $ \DR^r_{X/Y}(\sF,\nabla)\ten^{\oL}_{\sO_Y}\sO_X$ by the action of the ideal $\sI$ in $\DR(\sO_X/\sO_Y) \ten_{\sO_Y}\sO_X $ generated by elements of the form $\{a\ten 1 - 1 \ten a ~:~ a \in \sO_X\}$ and $\{da \ten 1 ~:~ a \in \sO_X\}$. 

Since the question is now local in nature, we may assume that $\sO_X$ is freely generated as a sheaf of graded-commutative  algebras
over $\sO_Y$ by elements $\xi_i \in \sO_{X,2}$. The ideal $\sI$ is then generated by elements of the form $\xi_i\ten 1 - 1 \ten \xi_i $ and $d\xi_i \ten 1$. Now, the differential in $\DR(\sO_X/\sO_Y) \ten_{\sO_Y}\sO_X $ maps $\xi_i\ten 1 -1\ten \xi_i$ to $(d\xi_i +\delta \xi_i) \ten 1 -  \delta \xi_i \ten 1 = d\xi_i \ten 1$ (since $\delta \xi_i$ has degree $1$, so lies in $\sO_Y$), and these elements act freely commutatively on $ \DR^r_{X/Y}(\sF,\nabla)\ten^{\oL}_{\sO_Y}\sO_X$, so the quotient map $\theta$ above is indeed a quasi-isomorphism.
\end{proof}

\begin{proposition}\label{DRrGorensteinProp}
 In the setting of Lemma \ref{DRrGorensteinLemma},  the functor  $f_*^{\sD}:=\DR^r_{X/Y}(-)$  from homotopy  right $\sD_X$-modules to homotopy  right $\sD_Y$-modules  has an exact  left adjoint $f^*_{\sD}$  which   on the underlying $\sO_Y$ and $\sO_X$-modules is quasi-isomorphic to the functor $\sE \mapsto \Omega^s_{X/Y}\ten_{\sO_Y}\sE[-s]$. The unit $\id \to f_*^{\sD}f^*_{\sD}$ of the adjunction is a quasi-isomorphism. 
\end{proposition}
\begin{proof}
 Because the generators are in even degree, this follows in much the same way as the corresponding statement for pushforwards and pullbacks of right $\sD$-modules along smooth affine morphisms. 

The left $\sD_X$-module $\cDiff(\sO_Y, \sO_X) \cong \sD_X\ten_{\sD_{X/Y}}\sO_X $ is quasi-isomorphic to the right de Rham complex $\DR^r_{X/Y}(\sD_X)$, which  the description above shows is perfect as a left $\sD_X$-module, the direct sum being finite. Its dual $\HHom_{\sD_X}( \DR^r_{X/Y}(\sD_X), \sD_X)$ is the right $\sD_X$-module
\[
 \DR_{X/Y}(\sD_X)= (\bigoplus_{p} \Omega^p_{X/Y}\ten_{\sO_X}\sD_X [-p], d)
\]
given by the left de Rham complex of $\sD_X$, the latter regarded as a left $\sD_X$-module.  The left adjoint to $f_*^{\sD}$ is thus given by the functor
\begin{align*}
f^*_{\sD}\sE:= &\sE\ten_{\sD_Y} \oR\hom_{\sD_X}(\cDiff(\sO_Y, \sO_X), \sD_X)\\
&\simeq \sE\ten_{\sD_Y}\DR_{X/Y}(\sD_X),
\end{align*}
where the homotopy left $\sD_Y$-module structure on $ \DR_{X/Y}(\sD_X)$ is induced from the right $\sD_Y$-module structure on $\cDiff(\sO_Y, \sO_X)$.

Our expression for $\DR_{X/Y}(\sD_X)$ amounts to giving a quasi-isomorphism $\DR_{X/Y}(\sD_X) \to \Omega^s_{X/Y}\ten_{\sD_{X/Y}}\sD_X[-s]  $, for the right $\sD_{X/Y}$-module structure on the line bundle $ \Omega^s_{X/Y}$ induced by the de Rham differential via the isomorphism $\Omega^s_{X/Y}\ten_{\sO_X}\sT_{X/Y}\cong \Omega^{s-1}_{X/Y}$. 
%
Since we have a natural isomorphism $\sD_Y\ten_{\sO_Y}\Omega^s_{X/Y}[-s] \to \Omega^s_{X/Y}\ten_{\sD_{X/Y}}\sD_X[-s]$,
it follows that for 
any right $\sD_Y$-module $\sE$, we have a quasi-isomorphism
\[
\sE\ten_{\sD_Y}\DR_{X/Y}(\sD_X)\to \sE\ten_{\sO_Y}\Omega^s_{X/Y}[-s],
\]
%
%
so we have constructed a natural quasi-isomorphism
\[
f^*_{\sD}\sE \to\Omega^s_{X/Y}\ten_{\sO_Y}\sE[-s].
\]

Using the formula of Lemma \ref{DRrGorensteinLemma}, we have
\[
 f^*_{\sD}f_*^{\sD}(\sF,\nabla)\simeq \Omega^s_{X/Y}[-s]\ten_{\sO_Y}\hom_{\sO_X}( \Omega^s_{X/Y},\sF)[s] \simeq \sF,
\]
so the unit of the adjunction is a quasi-isomorphism.
\end{proof}

\subsubsection{From quantisations to cohomology of $\omega_X$}

We now set about relating the right de Rham cohomology of $\omega_X$ to that of the determinant line bundle $\det \Omega^1_X$ on $X$. The determinant is a functor from perfect complexes to line bundles (i.e. dualisable perfect complexes), determined locally by the properties that $\det \sO_X =\sO_X$ and $\det \cone(\sG \to \sF)\simeq \det(\sG)^*\ten \det(\sF)$.

\begin{proposition}\label{DRrtoBMprop}
Take a dg $R$-manifold $X$ such that 
$\sO_X$ is  locally generated  in homological degrees $[0,2]$. Then there is a natural homotopy right $\sD_X$-module structure on the determinant line bundle $\det \Omega^1_X$. There is  a natural map 
\[
 \det \Omega^1_X[\dim X ] \to \omega_X
\]
in the derived category of homotopy right $\sD_X$-modules, where $\omega_X$ is the dualising complex,  
which induces a quasi-isomorphism
\[
 \DR^r_X(\det \Omega^1_X) \to \DR^r_X(\omega_X)[-\dim X],
\]
of right de Rham complexes.
\end{proposition}
In particular,  via Lemma \ref{BMlemma} this induces an isomorphism $\bH^i (X,\DR^r_X(\det \Omega^1_X)) \to \H^{BM}_{\dim X-i}(\pi^0X)$ to Borel--Moore homology of the underived truncation $\pi^0X$.  
\begin{proof}
The functor $ f_*^{\sD}:=\DR^r_{X/Y}(-)$ from Proposition \ref{DRrGorensteinProp} has a natural derived  right adjoint $f^!_{\sD}$, simply given by $\hom_{\sD_Y^{\op}}(\cDiff(\sO_Y, \sO_X),-)$, or on the underlying $\sO_Y$-modules by $\hom_{\sO_Y}(\sO_X,-)$; the dualising complexes then automatically satisfy $ f^!_{\sD}\omega_Y \simeq \omega_X$.

Applying Proposition \ref{DRrGorensteinProp} to the right $\sD_Y$-module $\omega_Y$, we have a quasi-isomorphism
$
\omega_Y \to f_*^{\sD}f^*_{\sD}\omega_Y
$
of homotopy right $\sD_Y$-modules, and in particular a homotopy inverse
\[
 f_*^{\sD}f^*_{\sD}\omega_Y \to \omega_Y. 
\]
By the adjunction above, this gives us a map 
\[
 f^*_{\sD}\omega_Y \to f^!_{\sD}\omega_Y\simeq\omega_X
\]
of  homotopy right $\sD_X$-modules. 

Letting $\sO_Y\subset \sO_X$ be the dg $\sO_{X,0}$-subalgebra generated by $\sO_{X,1}$,
on applying the functor $ f^*_{\sD}$ our map recovers the quasi-isomorphism $\omega_Y \simeq  f_*^{\sD}\omega_X$ from the proof of Lemma \ref{BMlemma}, and hence $\DR^r( f^*_{\sD}\omega_Y) \simeq \DR^r(\omega_X)$.

It only remains to show that $ f^*_{\sD}\omega_Y \simeq   \det \Omega^1_X[\dim X ]$. 
By Proposition \ref{DRrGorensteinProp}, we know that 
\[
 f^*_{\sD}\omega_Y \simeq \omega_Y\ten_{\sO_Y}\Omega^{\dim X-\dim Y}_X[\dim Y-\dim X]\simeq \omega_Y\ten_{\sO_Y}\det \Omega^1_{X/Y}[\dim X-\dim Y ], 
 \]
 and since 
 the dg manifold $Y$ is virtually LCI, it follows that  the dualising complex $\omega_Y$ is quasi-isomorphic to $ \det \Omega^1_{Y}[\dim Y]$, via the  trace map $\det \Omega^1_{Y}[\dim Y] \to  \det \Omega^1_{X^0}[\dim X^0]$. Explicitly, if we write $\sE:= \sO_{Y,1}$ with rank $r$ over $X^0$, then  as a graded algebra, $\sO_Y$ is the exterior algebra of $\sE$ over $\sO_{X^0}$, and its $\sO_{X^0}$-linear dual is $\det \sE^* \ten_{\sO_{X^0}}\sO_Y[-r] \cong \det \Omega^1_{Y/X^0}[-r]$, giving
\begin{align*}
 \omega_Y &\simeq \hom_{\sO_{X^0}}(\sO_Y, \omega_{X^0})\\
& \simeq  \omega_{X^0}\ten_{\sO_{X^0}} \det  \Omega^1_{Y/X^0}[-r]\\
&\simeq  (\det \Omega^1_{X^0})\ten_{\sO_{X^0}} \det  \Omega^1_{Y/X^0}[\dim X^0-r]\\
&\cong \det  \Omega^1_{Y}[\dim Y].
\end{align*}
\end{proof}

\begin{remark}\label{DRrGorensteinRmk}
 Although Proposition \ref{DRrtoBMprop} is phrased for dg schemes rather than derived schemes and appears to depend on the choice of model, there is a model independent formulation. If we have a morphism $g \co Y\to Y'$ of virtually LCI derived schemes, or even derived formal schemes, for which the map on reduced loci $(\pi^0Y)^{\red} \to (\pi^0Y')^{\red}$ is an isomorphism, then the functor $ g_*^{\sD}:=\DR^r_{X/Y}(-)$ gives a  quasi-equivalence on the corresponding categories of homotopy right $\sD$-modules, essentially by \cite{GaitsgoryRozenblyumCrystal} and \cite[Example \ref{DQvanish-stratDAex}]{DQvanish}. This implies that  the homotopy inverse functor is both a left and right adjoint  $g^*_{\sD}\simeq g^!_{\sD}$. Thus Proposition \ref{DRrGorensteinProp} remains valid if we replace $Y$ with any  virtually LCI derived  ind-scheme equipped with a map $X \to Y$ which is an isomorphism on reduced loci, provided we replace  $\Omega^s_{X/Y}[-s]$ with $\det \Omega^1_{X/Y}[\dim X-\dim Y]$.
 
 In particular, we can replace $Y$ with any of the formal schemes given by completing iterated products of $X^0$ along the diagonal copy of $\pi^0X$. Since these resolve the de Rham stack $X_{\dR}\co A \mapsto X(\H_0A^{\red})$ of $X$ and there is a canonical map $q \co X \to X_{\dR}$ (see \cite{simpsonHtpy}),   this allows us to interpret the construction of Proposition \ref{DRrGorensteinProp} as the left adjoint $q^*_{\sD} $ to the functor $q_*^{\sD} $ from right $\sD_X$-modules to right $\sD_{X_{\dR}}$-modules, for any derived Deligne--Mumford stack $X$ with cotangent complex generated in chain degrees $[0,2]$. On dualising complexes, we have $ q^*_{\sD}\omega_{X_{\dR}} \simeq \det \Omega^1_X[\dim X]$, while the right adjoint  $q^!_{\sD}$ sends $\omega_{X_{\dR}}$ to $\omega_X$. 
 
 The proof of Proposition \ref{DRrtoBMprop} will then adapt when phrased in terms of the canonical morphism $q \co X \to X_{\dR}$.
 
For an interpretation of these functors closer in spirit to the approach of \cite{GaitsgoryRozenblyumCrystal}, note that   \cite[Example \ref{DQvanish-stratDAex}]{DQvanish} effectively gives an interpretation of our right $\sD_X$-modules as corresponding to 
 coherent sheaves with respect to $!$-pullback on the derived stack $X_{\mathrm{strat}}\co A \mapsto \im(X(A) \to X(\H_0A^{\red}))$ of stratifications instead of living on $X_{\dR}$ as in \cite{GaitsgoryRozenblyumCrystal}. For the natural map $p \co X_{\mathrm{strat}} \to X_{\dR}$, there are pushforward and pullback functors $p_! \dashv p^!$, but the discussion above suggests that we also have a further left adjoint $p^* \dashv p_!$ under the conditions of Proposition \ref{DRrtoBMprop}.

 \end{remark}

\begin{corollary}\label{DRrtoBMcor}
Take a dg $R$-manifold $X$ such that 
$\sO_X$ is  locally generated  in homological degrees $[0,2]$. Then for any homotopy right $\sD_X$-module structure $\nabla$  on $\sO_X$, there is a corresponding homotopy left $\sD_X$-module structure on the determinant line bundle $\det \Omega^1_X$, and a natural map
\[
\sO_X[\dim X ] \to   (\det \Omega^1_X)^*\ten_{\sO_X} \omega_X
\]
in the derived category of  homotopy right $\sD_X$-modules,  inducing a quasi-isomorphism
\[
 \DR^r_X(\sO_X,\nabla)  \to \DR^r_X((\det \Omega^1_X)^*\ten_{\sO_X}\omega_X)[-\dim X],
\]
of right de Rham complexes.
\end{corollary}
\begin{proof}
 Since  $\det \Omega^1_X$ has a natural  homotopy right $\sD_X$-module structure by Proposition \ref{DRrtoBMprop}, a homotopy right $\sD_X$-module structure on $\sO_X$ automatically gives a homotopy left $\sD_X$-module structure on $(\det \Omega^1_X)^*=\hom_{\sO_X}(\det \Omega^1_X,\sO_X) $, with the same reasoning as in the classical case. Tensoring the map of Proposition \ref{DRrtoBMprop} with this left $\sD_X$-module then gives the desired map.
\end{proof}

\begin{corollary}\label{VFCCor}
 For $R$ discrete (i.e. $R=\H_0R$), take a  dg $R$-manifold $X$ such that 
$\sO_X$ is  locally generated  in homological degrees $[0,2]$, equipped with a non-degenerate $(-2)$-shifted Poisson structure $\pi$. 

There is then an essentially unique  flat right connection $\nabla$ on $\sO_X$  admitting quantisations  $S \in Q\cP(\sO_X, \nabla,-2)$ of $\pi$; these quantisations are  unique up to addition by locally constant functions $\hbar^2 R\brh$, and the space of such quantisations is discrete. 

To each such quantisation $S$, there is an associated  virtual fundamental class
\[
 [e^S] \in  \H^{- \dim X}(X,\DR^r_X((\det \Omega^1_X)^*\ten_{\sO_X}\omega_X))\brh
\]
in Borel--Moore homology with coefficients in $(\det \Omega^1_X)^*\brh$, equipped with a homotopy left $\sD$-module structure determined by $\nabla$. 
\end{corollary}
\begin{proof}
 Since $\pi$ is non-degenerate, the leading term $\pi_2$ induces a quasi-isomorphism $ \pi_2^{\flat} \co \Omega^1_X  \to \hom(\Omega^1_X, \sO_X)[2]$, and hence a quasi-isomorphism $ \det \Omega^1_X \simeq (\det \Omega^1_X)^*$ of line bundles.  This implies that  $u \co (\det \Omega^1_X)^{\ten 2} \simeq \sO_X$, so has a homotopy left $\sD_X$-module structure. This in turn induces a homotopy left $\sD_X$-module structure on $\det \Omega^1_X$, essentially given by setting $\nabla(\lambda):= \half\lambda^{-1}u^{-1}(du(\lambda^2))$.

Thus $\sO_X$ admits a homotopy right $\sD$-module structure, so the conditions of   Corollary \ref{uniqueconncor} are satisfied compatibly on open affine subschemes,  giving us a (possibly different) flat right connection $\nabla$ on $\sO_X$ and admitting   quantisations parametrised by $\hbar^2\H^0(X, \DR(X))\brh \cong \H^0(\pi^0X, \hbar^2 R\brh)$. The space of quantisations is discrete because $ \H^{<0}(X, \DR(X))=0$, so the homotopy groups are all trivial.

The map $Q\cP(A, \nabla,-2) \to 1+\hbar\H^0(X,\DR^r_X(\sO_X,\nabla))\brh$   from \S \ref{QME} sending $S$ to $[e^S]$  then combines with
Corollary \ref{DRrtoBMcor} to  give a class $[e^S]$ in twisted Borel--Moore homology.
\end{proof}
 
\begin{remark}\label{VFClocsysrmk} 
 When $R=\Cx$,  Corollary \ref{VFCCor}  combines with the argument of Lemma \ref{BMlemma} to give us a class in $\H^{BM}_{\dim X}(\pi^0X, \bM \brh)$, the Borel--Moore homology of the underived truncation $\pi^0X$ with coefficients in a rank $1$ local system $\bM$ with $\bM\ten_{\Cx}\sO_{\pi^0X}^{\an} \simeq  (\det \Omega^1_X)^*\ten_{\sO_X}\sO_{\pi^0X}^{\an}$. The   map from algebraic to analytic de Rham cohomology still exists with twisted coefficients, and although it need not be an isomorphism in general, it turns out that in this case it is, because Proposition \ref{EulerPropNew} and Remark \ref{QIMconnRmk} below  imply that $\bM$ is a $\mu_2$-torsor, so locally trivial in the  \'etale topology.
\end{remark}

\begin{remark}\label{DRrtoBMRmk}
 Although Corollary \ref{VFCCor} is phrased for dg schemes, all we really need is a setting where Proposition \ref{DRrtoBMprop} holds, providing us with a canonical flat right connection on   $\det \Omega^1_X$ and a morphism $\det \Omega^1_X \to \omega_X[-\dim X] $ inducing a quasi-isomorphism on right de Rham complexes.  This should work for any $(-2)$-shifted symplectic derived Deligne--Mumford stack, since by \cite[Theorem 5.18]{BBBJdarboux} or \cite[\S 3.6]{BouazizGrojnowski}, these are all  locally of the form in Corollary \ref{VFCCor},  
and Remark \ref{DRrGorensteinRmk} explains how to establish the required maps of right $\sD$-modules in that generality.

Lemma \ref{QIMconnlemma} and Remark \ref{QIMconnRmk} below then say that the  homotopy left $\sD$-module structure on $(\det \Omega^1_X)^*$ determined by $\nabla$ in Corollary \ref{VFCCor} is just the essentially unique such structure  compatible with the orthogonal  inner product induced by the first term $\pi_2$ of the Poisson structure.
\end{remark}

%

\subsection{Strict Poisson structures}


\subsubsection{Quantisations of strict Poisson structures}

 \begin{definition}
 Say that a $(-2)$-shifted Poisson structure $\pi$ on a CDGA $A$ is \emph{strict} if $\pi=\pi_2$. In particular, this makes $A[-2]$ a DGLA rather than just an $L_{\infty}$-algebra.

Say that a    $(-2)$-shifted Poisson structure $\pi$ is strictly non-degenerate if the map $\pi_2^{\flat}\co \Omega^1_A \to   \HHom_A(\Omega^1_{A},A)[2]$ is an isomorphism (not just a  quasi-isomorphism).
\end{definition}

%

\begin{lemma}\label{StrictPoissonLemmaNew}
 Let $X$ be a dg $R$-manifold equipped with a strictly non-degenerate strict $(-2)$-shifted Poisson structure $\pi$. Then  $\pi$ determines  a symmetric inner product $Q$  on the vector bundle $\sE:=\sO_{X,1}$ on $X^0$ and a global section $\phi \in \Gamma(X^0,\sE)$ with $dQ(\phi,\phi)=0 \in \Gamma(X^0, \Omega^1_{X^0})$. The underived truncation $\pi^0X \subset X^0$ is the vanishing locus of $\phi$.

Moreover, $\sO_X$ is locally generated in homological degrees $[0,2]$, and the Poisson structure $\pi$  determines an isomorphism $\sO_{X,2}/\L^2\sE \cong \sT_{X^0}$. If they exist, sections of the surjection $\sO_{X,2} \to \sT_{X^0} $ then correspond to (not necessarily flat) left connections $\nabla_{\sE}$ on $\sE$ which are compatible with $Q$.
\end{lemma}
\begin{proof}
Giving such a Poisson structure on $X$ is equivalent to defining a shifted Lie bracket $\{-,-\}\co \sO_X^{\ten 2} \to \sO_X[2]$ which respects the differential $\delta$,  is a biderivation with respect to the multiplication on $\sO_X$, and satisfies a strict non-degeneracy condition. The symmetric inner product $Q$ is then just the restriction of $\{-,-\} $ to $\sE$, while non-degeneracy of $Q$ implies that the differential $\delta \co \sE \to \sO_{X^0}$  is given by $Q(\phi,-)$ for a unique element  $\phi \in \Gamma(X^0,\sE)$. It also implies  that the sheaf of functions on the vanishing locus of $\phi$ is just  $\sO_{X^0}/\delta\sE$, but this is $\sO_{\pi^0X}$.  

For degree reasons, we must have $\{\sO_{X,0},\sE\}=0$ and hence $\{\sO_{X,0},\L^2\sE\}=0 $, for $\L^2\sE \subset \sO_{X,2}$. The map $\sO_{X,2} \to \sT_{X^0}$ sending an element $z$ to the derivation $\{z,-\}$ thus descends to a map $\sO_{X,2}/\L^2\sE \to \sT_{X^0}$; strict  non-degeneracy of $\pi$ implies that this is an isomorphism and that there are no generators in degrees higher than $2$. 

Given a section $\sigma \co \sT_{X^0} \to \sO_{X,2}$ of that surjection, there is a left connection $\nabla_{\sE} \co \sE \to \sE\ten_{\sO_{X,0}}\Omega^1_{X^0}$ determined by the property that for $v \in \sT_{X^0}$ and $e \in \sE$, we have $v \lrcorner \nabla_{\sE}(e)=\{v,e\} \in \sE$; this satisfies $dQ(e,e')= Q(\nabla_{\sE}e,e')+ Q(e, \nabla_{\sE}e') \in \Omega^1_{X^0}$. The other choices of section are given by $\sigma + \alpha \lrcorner$ for elements $\alpha  \in \L^2\sE\ten_{\sO_{X,0}}\Omega^1_{X^0} $. The connection corresponding to $\alpha$ is then $\nabla_{\sE}+ Q(\alpha,-)$, so all orthogonal left connections arise in this way, since $Q$ gives an isomorphism from $\L^2\sE$  to the space of antisymmetric endomorphisms of $\sE$.

We next observe that the differential $\delta$ on $\sO_X$ must be given by $\{\phi,-\}$. It suffices to check this on generators, but both maps are automatically $0$ on $\sO_{X,0}$, and they agree on $\sO_{X,1}$ by definition, so we need only calculate the effect on $\sO_{X,2}$. For all $z \in \sO_{X,2}$ and $e \in \sE=\sO_{X,1}$, we have
\begin{align*}
 \{\{\phi,z\},e\} &= \{\phi,\{z,e\}\} - \{z,\{\phi,e\}\}\\
&=\delta\{z,e\} - \{z,\delta e\}\\
&= \{\delta z,e\},
\end{align*}
so  $\delta z - \{\phi,z\}$ is central for $\{-,-\}$, and must therefore be $0$ by non-degeneracy of $Q$.

Now, since $\delta^2=0$, it follows that $\{\phi,\{\phi,-\}\}=0$, or equivalently that the element $\{\phi,\phi\}=Q(\phi,\phi)$ is central. This amounts to saying that $\{Q(\phi,\phi),z\}=0$ for all $z \in \sO_{X,2}$; using the isomorphism  $\sO_{X,2}/\L^2\sE \cong \sT_{X^0}$ above, this can be rephrased as $dQ(\phi,\phi)=0 \in \Gamma(X^0, \Omega^1_{X^0}) $.
\end{proof}

%


\begin{lemma}\label{strictcompatconn} 
Let $X$ be a dg $R$-manifold equipped with a strictly non-degenerate strict $(-2)$-shifted Poisson structure $\pi$.
Then   there exists an essentially unique right connection $\nabla$ on $\sO_X$ satisfying the conditions of Proposition \ref{uniqueconn}.

If $R$ is discrete (i.e. $R\simeq \H_0R$) and $\pi^0X$ connected, then the resulting space  
\[
\oR\Gamma(X^0,Q\cP(\sO_X, \nabla,-2))
\]
of quantisations  of $\pi_{\hbar}\in \Gamma(X^0,\cP(\sO_X,-2))$ with respect to $\nabla$ is given by  the discrete set  
\[
\hbar \pi + \hbar^2 R\brh,
\]
up to weak equivalence.
\end{lemma}
\begin{proof}
 The right connection $\nabla$ on $\sO_X$ from Proposition \ref{uniqueconn} is determined by the condition $\nabla \pi=0$. Since $\nabla(\alpha \lrcorner \pi) = d\alpha \lrcorner \pi \mp \alpha \lrcorner \nabla(\pi)$ for all $\alpha \in \Omega^1_X$,  we then define
 $\nabla \co \sT_X \to \sO_X$  by
\begin{align*}
 \nabla(v)&= \nabla( (\pi^{\flat})^{-1}(v) \lrcorner \pi)\\
&=  d((\pi^{\flat})^{-1}(v)) \lrcorner \pi,
\end{align*}
for $\pi^{\flat} \co \Omega^1_X \to \sT_X[2]$ the isomorphism given by contraction. Since this is flat and satisfies $\nabla \pi=0$, it must be the essentially unique connection of Proposition \ref{uniqueconn}, up to coherent homotopy.
  
Next, observe that since the right connection $\nabla$  has no higher terms (a consequence of the Poisson structure being strict), Lemma \ref{DRrBVlemma} ensures that the associated $L_{\infty}$ structure $\{[-,-]_{\nabla,n}\}_n$ on the complex of polyvectors is just the Schouten--Nijenhuis bracket, giving us a natural filtered DGLA isomorphism 
\[
Q\widehat{\Pol}(\sO_X,\nabla,-2)[-1] \simeq \widehat{\Pol}(\sO_X,\nabla,-2)[-1]\brh.
\]
In particular, inclusion of constants gives us a natural map
\[
\cP(\sO_X, \nabla,-2) \to  Q\cP(\sO_X, \nabla,-2),
\]
so $\pi_{\hbar}=\hbar \pi$ is a natural quantisation of itself with respect to $\nabla$.

By Corollary \ref{uniqueconncor},  the set $\hbar \pi+\hbar^2 R\brh$ is thus contained the space of quantisations, with the inclusion map being a weak equivalence.
\end{proof}

\subsubsection{From quantisations to fundamental classes}

\begin{lemma}\label{QIMconnlemma}
Let $X$ be a dg $R$-manifold equipped with a strictly non-degenerate strict $(-2)$-shifted Poisson structure $\pi$, and let $\sE:=\sO_{X,1}$ with   $\sO_Y \subset \sO_X$ the subalgebra generated by $\sO_{X^0}$ and $\sE$. Then for the right connection $\nabla$ on $\sO_X$ given by Lemma \ref{strictcompatconn}, we have a natural quasi-isomorphism
\[
 (\det \sE) \ten_{\sO_{X^0}}  \omega_Y[-\dim X] \to  \DR^r_{X/Y}(\sO_X,\nabla)
\]
of  homotopy right $\sD_Y$-modules,  
where the right $\sD_Y$ -module structure on the left-hand side combines the canonical structure on $\omega_Y$ with the unique flat left connection on $\det \sE$ compatible with the inner product $Q$ of Lemma \ref{StrictPoissonLemmaNew}. 

If they exist, all choices of (not necessarily flat) orthogonal  left connection $\nabla_{\sE}$ on $\sE$ give rise to homotopy inverses of the quasi-isomorphism.
\end{lemma}
\begin{proof}
As in the proof of Proposition \ref{DRrtoBMprop}, we have $\omega_Y \cong  (\det \sE^*)\ten_{\sO_{X^0}} \det \Omega^1_{X^0} [\dim Y]  $, so
\[
(\det \sE) \ten_{\sO_{X^0}}  \omega_Y[-\dim X] \cong (\det \Omega^1_{X^0})\ten_{\sO_{X^0}}\sO_Y [-\dim X^0]
\]
since strict non-degeneracy of $\pi$ implies that $\dim X^0 = \dim X - \dim Y$.

For degree reasons, the Poisson bracket on $\sO_X$ satisfies  $\{\sO_{X^0},\sO_Y\}=0$, so contraction with $\pi$ gives a map $\Omega^1_{X^0} \to \sT_{X/Y}$, and for $a,b \in \sO_{X^0}$, we have $[da \lrcorner \pi,db \lrcorner \pi]=0$ for the canonical Lie bracket $[-,-]$ on $\sT_{X/Y}$. The proof of Lemma \ref{strictcompatconn} gives $\nabla(da \lrcorner \pi) = d^2a \lrcorner \pi=0$, so we can deduce that any product of elements of the form $da \lrcorner \pi$ lies in the kernel of $\nabla$, since $\DR^r_{X/Y}(\sO_X,\nabla)$ is a BV algebra. 

The same is automatically true for any $\sO_Y$-linear combination, so contraction with $\pi$ gives a chain map (cf. Definition \ref{QPolmudef}) 
\[
 \mu(-,\pi) \co (\Omega^*_{X^0}\ten_{\sO_{X^0}}\sO_Y,0) \to \DR^r_{X/Y}(\sO_X,\nabla) \subset \DR^r_{X}(\sO_X,\nabla),
\]
  and in particular a map $  (\det \Omega^1_{X^0})\ten_{\sO_{X^0}}\sO_Y [-\dim X^0] \to \DR^r_{X/Y}(\sO_X,\nabla)$ of $\sO_Y$-modules. This must be a quasi-isomorphism because Lemma \ref{DRrGorensteinLemma} implies that it becomes so on pulling back to $X$. 

By Lemma \ref{StrictPoissonLemmaNew}, orthogonal left connections on $\sE$ correspond to  subspaces $\sF$ complementary to $\L^2\sE \subset \sO_{X,2}$, and then similarly to the proof of   Lemma \ref{DRrGorensteinLemma}, we obtain a homotopy inverse by taking the quotient by the action of the ideal  in $\DR(\sO_X/\sO_Y) $ generated by $\sF$ and its image under the differential $d\pm \delta$.  

It remains to show that our map $ (\det \sE) \ten_{\sO_{X^0}}  \omega_Y[-\dim X]\to \DR^r_{X/Y}(\sO_X,\nabla)$ is compatible with  the respective right $\sD_Y$-module structures. To show this, it suffices to restrict to local co-ordinates. For $n:= \dim X^0$, without loss of generality we may replace $X^0$ with an  \'etale neighbourhood $U^0$  admitting an \'etale map $U^0 \to \bA^n$ such that $\sE|_{U^0}$ is free as an orthogonal bundle. This  gives co-ordinates $x_1, \ldots, x_n$ in $\sO_{U^0}$, and an orthonormal  basis $e_1, \ldots, e_r$ for $\sE|_{U^0}$. Then  the $dx_i$ form a basis for $\Omega^1_{U^0}$, giving a  dual basis $\{\pd_{x_i}\}_{i=1}^n$ of $\sT_{U^0}$.    Letting $\nabla_{\sE}$ be the  left connection on $\sE|_{U^0}$ which kills the elements $e_i$, Lemma \ref{StrictPoissonLemmaNew} gives a map $\sT_{U^0}\to \sO_{X,2}|_{U^0}$, and we let $\xi_i$ be the image of $\pd_{x_i}$.   

In these co-ordinates, the right $\sD_Y$-module structure on $\omega_Y$ is determined by the property that the tangent vectors $\pd_{x_i}$, $\pd_{e_j}$ act trivially on the element  
\[
dx_1 \wedge \ldots \wedge dx_n \ten (e_1\wedge \ldots \wedge e_r)^* \in \Omega^n_{U^0}\ten_{\sO_{X^0}} (\det\sE)^*.
\]
Since $\nabla_{\sE}$ induces the unique orthogonal left connection on $\det \sE|_{U^0}$, the element $e_1\wedge \ldots \wedge e_r$ is horizontal, so our right $\sD_Y$-module structure on $(\det \sE) \ten_{\sO_{X^0}}\omega_Y|_{U^0} \simeq \Omega^n_{U^0}[n-r] $ has the basis tangent vectors $\pd_{x_i}$, $\pd_{e_j}$ acting trivially on $dx_1 \wedge \ldots \wedge dx_n$.

Our expression for $\pi$ in these co-ordinates  is just $\sum_{i=1}^n \pd_{x_i}\pd_{\xi_i} + \half\sum_{j=1}^r \pd_{e_j}\pd_{e_j} $, and  the flat right  connection $\nabla$ on $\sO_X$ killing this just has  the basis tangent vectors $\pd_{x_i}$, $\pd_{\xi_i}$, $\pd_{e_j}$ acting trivially on $1$, so polynomials in this basis are cocycles in $\DR^r_{X/Y}(\sO_X,\nabla)  $. The chain map $\mu(-,\pi)$ then sends $dx_1 \wedge \ldots \wedge dx_n $ to $\pd_{x_1} \ldots \pd_{x_n}$, and basis tangent vectors act trivially on both for the respective structures, so we indeed have a map of right $\sD_Y$-modules.
\end{proof}

This leads to the following statement for far more general $(-2)$-shifted symplectic structures: 
\begin{proposition}\label{QIMconnpropDM}
Take a $(-2)$-shifted symplectic derived Deligne--Mumford stack over a discrete base $R=\H_0R$, together with a right $\sD$-module structure on $\det \Omega^1_X$ locally agreeing with that of Proposition \ref{DRrtoBMprop}. For the unique left $\sD$-module structure on  $\det \sT_X$ compatible with the leading term $\pi_2$ of the Poisson structure, the right $\sD$-module $(\sO_X,\nabla)$ of Proposition \ref{uniqueconn} then admits a 
 a canonical quasi-isomorphism
  \[
 \phi \co  (\sO_X,\nabla) \to (\det \sT_X)\ten_{\sO_X}(\det \Omega^1_X)
  \]
of right $\sD$-modules on $X$, such that 
\[
 \H_0\phi \co \sO_{\pi^0X} \to \H_0(\det \sT_X)\ten_{\sO_{\pi^0X}}\H_0(\det \Omega^1_X) 
\]
is the obvious isomorphism.
\end{proposition}
\begin{proof}
As in Remark \ref{DRrtoBMRmk}, we know that the conditions of Lemma \ref{QIMconnlemma} are satisfied locally on any $(-2)$-shifted symplectic derived DM stack. If we apply the functor $f^*_{\sD}$ of Proposition \ref{DRrGorensteinProp} locally to \'etale neighbourhoods $U$, then Lemma \ref{QIMconnlemma} gives us a quasi-isomorphism $\phi_U$ of right $\sD$-modules on $U$ between  $(\sO_U,\nabla)$ and the tensor product  $(\det \sE) \ten_{\sO_{U^0}}  \det \Omega^1_U$, where 
 $\det \sE$ is given the unique left $\sD$-module structure compatible with the inner product $(\det\sE)\ten_{\sO_{X^0}}(\det\sE) \cong \sO_{X^0}$. 
 
 The non-degenerate inner product on the cotangent complex $\Omega^1_X[-1]$ induced by the first term $\pi_2$ of the Poisson structure automatically induces an orthogonal inner product on its determinant line bundle $\det \sT_X$, 
with an essentially unique compatible left connection. Locally, we can rewrite $\det \sE$ as $\det \sT_U$, and this connection corresponds to the connection on $\det \sE$ above. Thus we have established our desired quasi-isomorphism \'etale locally, and we will now see that this guarantees a global quasi-isomorphism.
 
 By the argument of \cite{FeiginTsygan}, the derived de Rham complex of $\sO_X$ is quasi-isomorphic to the algebraic de Rham complex  \cite{HartshorneAlgDeRham} of $\pi^0X$, so its  negative cohomology vanishes and its cohomology in degree $0$ is given by constants $R$. Therefore the kernel $\ker(\DR(\sO_X) \to \H_0\sO_X)$ has only strictly positive cohomology. For any line bundle $\sL$ on $X$ with right $\sD$-module structure,  the simplicial sheaf of derived $\sD$-module automorphisms of $\sL$ which fix the $\H_0\sO_X$-module $\H_0\sL$ is just given via exponentiation by Dold--Kan denormalisation of the good truncation chain complex $\tau_{\le 0} \ker(\DR(\sO_X) \to \H_0\sO_X)$. Since this complex is acyclic, the space of derived automorphisms is contractible, so descent is straightforward. 

 Taking $\sL:=(\sO_X,\nabla)$ in the previous paragraph, we deduce that  the maps $\phi_U$ must agree canonically on overlaps up to coherent homotopy,    combining to give our required global quasi-isomorphism 
$\phi$ of $\sD$-modules with the additional property that $\H_0\phi$ 
is the canonical isomorphism. 
\end{proof}

\begin{remarks}\label{QIMconnRmk}

Note that Remark \ref{DRrGorensteinRmk} explains how $\det \Omega^1_X$ should always carry a canonical right $\sD$-module structure, which would make the hypothesis of Proposition \ref{QIMconnpropDM} redundant. It is interesting to note that the the characterisation $ (\det \sT_X)\ten_{\sO_X} (\det \Omega^1_X)\simeq (\sO_X,\nabla)$  then means that the right connection of Proposition \ref{uniqueconn} depends only on the leading term $\pi_2$ of the Poisson structure.

Beware the map of $\sO_X$-modules underlying $\phi$ in Proposition \ref{QIMconnpropDM}  will not usually be the canonical quasi-isomorphism $\sO_X \simeq  (\det \sT_X)\ten_{\sO_X}(\det \Omega^1_X)$, since the  proof does not work with the complex $\ker(\DR(\sO_X) \to \sO_X)$ in place of  $\ker(\DR(\sO_X) \to \H_0\sO_X)$. This means that the two right $\sD$-module structures on $\sO_X$ which we encounter can differ, but that they will be conjugate by an element of $\bH^0(X,\sO_X)$ lifting $1 \in \H^0(\pi^0X,\sO_{\pi^0X})$. 
\end{remarks}

 \begin{proposition}\label{kappaprop} 
Let $X$ be a dg $R$-manifold equipped with a strictly non-degenerate strict $(-2)$-shifted Poisson structure, and let $\sE:=\sO_{X,1}$ with   $\sO_Y \subset \sO_X$ the subalgebra generated by $\sO_{X^0}$ and $\sE$. Then for the right connection $\nabla$ on $\sO_X$ given by Lemma \ref{strictcompatconn}, and for any left connection $\nabla_{\sE}$ on $\sE$, consider the associated homotopy inverse map 
\[
\Psi_{\nabla_{\sE}} \co   \DR^r_{X}(\sO_X,\nabla)\to \DR^r_Y((\det \sE) \ten_{\sO_{X^0}}  \omega_Y)[-\dim X]
\]
of Lemma \ref{QIMconnlemma}.

The image under this map of the cocycle $e^{\hbar \pi} \in \z^0\DR^r_{X}(\sO_X,\nabla)\brh$ associated to the quantisation $\hbar\pi$ of  Lemma \ref{strictcompatconn} can be expressed in terms of a local orthonormal basis $\{e_j\}_j$ for $\sE$ as
\[
\Psi_{\nabla_{\sE}}(e^{\hbar \pi} )= \hbar^{\dim X^0}  \exp( \hbar^{-1}Q(\kappa) + \hbar^{-1}\nabla_{\sE}(\phi)+ \half \hbar\sum_j \pd_{e_j}^2   + \sum_j \nabla_{\sE}(e_j)\pd_{e_j} ),
\]
for $Q$ 
and $\phi$ as in Lemma \ref{StrictPoissonLemmaNew},
and $\kappa$ the curvature of the connection $\nabla_{\sE}$. 

The image of this cocycle in 
\[
\H^{-\dim X}\DR^r_{X^0/R}((\det \sE)\ten_{\sO_{X^0}}\omega_{X^0})\brh\cong\H^{\rk(\sE)}\DR(X^0,\det \sE )\brh
\]
is  the cohomology class 
\[
 [\exp(\hbar \pi)] \mapsto  \begin{cases} \frac{\hbar^{(\dim X)/2}[Q(\kappa)^{\rk(\sE)/2}]}{(\rk(\sE)/2)!}  & 2\mid \dim(X)\\ 
								      0 & 2 \nmid \dim X.
                       \end{cases}
\]
\end{proposition}
\begin{proof}
Choose local co-ordinates $\{x_i\}_{i=1}^n$ on $X^0$, with the connection $\nabla_{\sE}$ then determining co-ordinates $\xi_i \in \sO_{X,2}$ with $\nabla_{\sE}(e)= \sum_i\{\xi_i,e\}dx_i$ for $e \in \sE$. Also choose   orthonormal basis vectors $e_j$ locally for $\sE$, but note that since we are fixing $\nabla_{\sE}$ at the outset, we cannot assume that these are horizontal. In these co-ordinates, the Poisson structure is given by 
\[
 \pi = \sum_{i=1}^n \pd_{x_i}\pd_{\xi_i} + \half\sum_{j=1}^r \pd_{e_j}\pd_{e_j}+ \sum_{ijl} c_{jli}e_l \pd_{e_j}\pd_{\xi_i}+\sum_{ikjl} \lambda_{jl}^{ik} e_je_l\pd_{\xi_i}  \pd_{\xi_k},
 \]
 where $\nabla_{\sE}(e_j)= \sum_{il} c_{jli}e_l dx_i $ and 
 $Q(\kappa) =  \sum_{ikjl} \lambda_{jl}^{ik} e_je_l dx_i\wedge dx_k$.

Now  observe that for $I=\prod_i \pd_{x_i} \pd_{\xi_i}$, we have
\begin{align*}
 \prod_{j\in J} \pd_{\xi_j}\exp(\hbar\sum_{i=1}^n \pd_{x_i} \pd_{\xi_i})&=\prod_{j\in J} \pd_{\xi_j}\prod_{i=1}^n (1+ \hbar \pd_{x_i} \pd_{\xi_i})\\
&= \prod_{j\in J} \pd_{\xi_j} \prod_{i\notin J} (1+ \hbar \pd_{x_i} \pd_{\xi_i})\\
&=  (\prod_{j\in J} dx_j \prod_{i\notin J} (dx_id\xi_i +\hbar)) \lrcorner I\\
&= \hbar^n    (\prod_{j\in J} \hbar^{-1} dx_j \prod_{i\notin J} (\hbar^{-1}dx_id\xi_i +1)) \lrcorner  I\\ 
&= ( \prod_{j\in J}\hbar^{-1}dx_j) \exp(\hbar^{-1}\sum_{i=1}^n dx_i d\xi_i) )\lrcorner  \hbar^nI.   
\end{align*}

Expanding and contracting then allows us to rewrite our class $\exp(\hbar \pi)$ as
\[
  \exp( \half \hbar \sum_{j}\pd_{e_j}^2 +\sum_{ijl} c_{jli}e_l \pd_{e_j}dx_i+ \hbar^{-1}\sum_{ikjl} \lambda_{jl}^{ik} e_je_l dx_i dx_k) \exp(\hbar^{-1}\sum_i dx_i d\xi_i)\lrcorner  \hbar^nI 
\]

The homotopy inverse $\Psi_{\nabla_{\sE}}$ associated to $\nabla_{\sE}$ kills all terms in the image of $(d\xi_i +\delta \xi_i) \lrcorner -$, giving 
\[
   \exp(\hbar^{-1}\sum_i dx_i d\xi_i) \lrcorner \beta \mapsto \exp(\hbar^{-1}\sum_i\delta \xi_i dx_i ) \lrcorner \beta =\exp(\hbar^{-1}\nabla_{\sE}(\phi))\lrcorner \beta 
\]
for any $\beta$, using Lemma \ref{StrictPoissonLemmaNew}. The rest then simplifies to give the globally defined expression
\[
 \Psi_{\nabla_{\sE}}( e^{\hbar \pi})=\exp( \hbar^{-1}Q(\kappa) + \hbar^{-1}\nabla_{\sE}(\phi)+\half \hbar\sum_j \pd_{e_j}^2   + \sum_j \nabla_{\sE}(e_j)\pd_{e_j}) \hbar^n,
\]
where we have dropped the factor $I$ because $\mu(dx_1\wedge \ldots \wedge dx_n,\pi)= \prod_i \pd_{\xi_i} $.
\end{proof}

\subsection{Relating virtual fundamental classes and Euler classes}

For dg manifolds with strictly non-degenerate strict Poisson structures, we now give a more explicit description of  the  virtual fundamental classes of Corollary \ref{VFCCor}. 

\subsubsection{Euler and Thom classes}\label{eulersn}

\begin{lemma}\label{eulerlemma}
 If $\det_{\R}$ denotes the determinant representation of the orthogonal group $\OO_r(\R)$, then in low degrees the relative cohomology groups $\H^i(B\OO_r(\R),B\OO_{r-2}(\R);\det_{\R})$ are given for $r$ even by:
 \[
  \H^i(B\OO_r(\R),B\OO_{r-2}(\R);{\det}_{\R}) \cong  \begin{cases} \R.e_r & i=r \\ \R.e_{r-2} & i =r-1 \\ 0 & i<r-1, \end{cases}
 \]
 where $e_p \in \H^i(B\SO_p,\Z)$ denotes the Euler class. 
For $r$ odd,  the cohomology groups vanish 
for all $i<2r-2$.
 \end{lemma}
 \begin{proof}
 Observe that we have a fibration sequence $B\SO_r(\R) \to B\OO_r(\R) \to BC_2$, and that the resulting representation $\R.C_2$ of $\OO_r$ is just $\R \oplus \det_{\R}$. We  therefore have  $ \oR\Gamma(\OO_r, \R) \oplus \oR\Gamma(\OO_r,\det_{\R}) \simeq \oR\Gamma(\SO_r, \R)$.
  
It then follows for instance  from \cite[Theorems 1.5 and 1.6]{brownCohoBSOnBOn} that $\oR\Gamma(\OO_r,\det_{\R})$ is a free $\oR\Gamma(\OO_r,\R)$-module generated by $e_r \in \H^r(B\SO_r,\Z)$ 
when $r$ is even, and is zero for $r$ odd. Since $\H^*(\OO_r,\R)$ is concentrated in even degrees, 
the long exact sequence for $\H^*(B\OO_r,B\OO_{r-2};\det_{\R})$ splits, with the maps
\begin{align*}
  \H^{2k}(B\OO_{r-2},{\det}_{\R}  ) \to &\H^{2k+1}(B\OO_r,B\OO_{r-2};{\det}_{\R})\text{ and }\\ 
 &\H^{2k}(B\OO_r,B\OO_{r-2};{\det}_{\R})\to  \H^{2k}(B\OO_r,{\det}_{\R}) 
\end{align*}
 being isomorphisms for all $k$, giving the description required.
\end{proof}

\begin{remark}
We will refer to the canonical generator $e_r$ of $\H^r(B\OO_r(\R),B\OO_{r-2}(\R);\det_{\R})$ as the Euler class, since its image in $ \H^r(B\OO_r(\R),\det_{\R})$ is the Euler class. However, the image of $e_r$  in $ \H^r(B\OO_r(\R),B\OO_{r-1}(\R);\det_{\R})$ is usually known as the Thom class, so we could equally well refer to  $e_r$ as such.
\end{remark}

\begin{definition}\label{Qrdef}
 Define  the affine scheme $Q_r \subset \bA^r$ to be vanishing locus of $\sum_{j=1}^r z_j^2$; observe that this is equivariant under the action of the orthogonal group $\OO_r$. Define the quasi-affine scheme $Q_r^*$ by $Q_r^*:=Q_r\setminus \{0\}$.
 \end{definition}

\begin{lemma}\label{stiefellemma}
For the analytic topologies on $Q_r(\Cx)$ and $\OO_r(\Cx)$, the homotopy quotient topological spaces $[Q_r(\Cx)/\OO_r(\Cx)]$ and $[Q_r^*(\Cx)/\OO_r(\Cx)]$ are homotopy equivalent to $B\OO_r(\R)$ and $B\OO_{r-2}(\R)$, respectively.
 \end{lemma}
 \begin{proof}
  The space $Q_r(\Cx)$ has an obvious deformation retraction to $\{0\}$ by scaling, so the natural map $ [Q_r(\Cx)/\OO_r(\Cx)] \to B\OO_r(\Cx)$ is a homotopy equivalence. Moreover, $\OO_r(\R) \subset \OO_r(\Cx)$ is a deformation retract, so  $B\OO_r(\R) \to B\OO_r(\Cx)$ is also a homotopy equivalence.
  
 If we write $z_j:=x_j+iy_j$, then $Q_r(\Cx) \subset \R^{2s}$ consists of points satisfying $\sum_{j=1}^r x_j^2 = \sum_{j=1}^ry_j^2  $ and $\sum_j x_jy_j=0$. In other words, $\|\uline{x}\|= \|\uline{y}\|$ and   $\uline{x}\cdot \uline{y}=0$. Scaling gives a deformation retraction of $Q_r^*(\Cx)$ onto its subspace $V_2(\R^r)$ of such points with $\|\uline{x}\|=1$ (known as the Stiefel $2$-manifold), which is naturally homeomorphic to $\OO_r(\R)/\OO_{r-2}(\R)$. Thus we have homotopy equivalences $ [Q_r^*(\Cx)/\OO_r(\Cx)] \xla{\sim} [V_2(\R^r)/ \OO_r(\R)] \cong B\OO_{r-2}(\R)$.
 \end{proof}

\begin{remark}
It is interesting to note that the homotopy equivalence $[Q_r^*(\Cx)/\OO_r(\Cx)]\to B\OO_r(\Cx)$ does lift to a map  $[Q_r^*/\OO_r] \to B\OO_{r-2}$ of algebraic stacks. An object of $[Q_r^*/\OO_r]$ consists of an orthogonal  rank $r$ vector bundle $\sE$ together with a non-zero section $\phi$ satisfying $\phi\cdot \phi=0$, and this gives rise to a vector bundle $\<\phi\>^{\perp}/\<\phi\>$ of rank $r-2$, and together these give a map $[Q_r^*/\OO_r] \to B\OO_{r-2}$ . 

However, this is not a map of stacks over $B\OO_r$. Such a map would amount to giving an $\OO_r$-equivariant map $Q_r^* \to  \OO_r/\OO_{r-2} =:V_2$ of schemes lifting the homotopy equivalence; this is  impossible because $V_2$ is affine and $Q_r^*$ has affine closure $Q_r$ (for $r>1$), which is contractible. Thus there is no  map  $([Q_r/\OO_r], [Q_r^*/\OO_r]) \to (B\OO_r, B\OO_{r-2})$ of morphisms of algebraic stacks lifting the homotopy equivalence $([Q_r(\Cx)/\OO_r(\Cx)], [Q_r^*(\Cx)/\OO_r(\Cx)]) \to (B\OO_r(\Cx), B\OO_{r-2}(\Cx))$
of morphisms of topological spaces. 
\end{remark}

 Combining Lemmas \ref{eulerlemma} and \ref{stiefellemma} gives:
 \begin{corollary}\label{stiefelcor}
  If $\det_{\Cx}$ denotes the determinant representation of the orthogonal group $\OO_r(\Cx)$, then for $r$ even we have
 \[
 \H^i([Q_r(\Cx)/\OO_r(\Cx)],[Q_r^*(\Cx)/\OO_r(\Cx)] ;{\det}_{\Cx})  \cong  \begin{cases} \Cx.e_r & i=r \\ \Cx.e_{r-2} & i =r-1 \\ 0 & i<r-1,\end{cases}
 \]
 for Euler classes $e_p$,
 with the natural map $\H^r([Q_r(\Cx)/\OO_r(\Cx)],[Q_r^*(\Cx)/\OO_r(\Cx)] ;\det_{\Cx}) \to \H^r( B\OO_r(\Cx),\det_{\Cx})  $  being an isomorphism.
 
For $r$ odd, we have $\H^i([Q_r(\Cx)/\OO_r(\Cx)],[Q_r^*(\Cx)/\OO_r(\Cx)];\det_{\Cx})=0$ for all $i<2r-2$.
\end{corollary}

\subsubsection{Twisted de Rham complexes}

For the affine scheme $Q_r \subset \bA^r$ of Definition \ref{Qrdef}, 
we now consider the na\"ive de Rham complex $\DR(\sO_{Q_r}/R)$ given by $(\Omega^*_{Q_r/R},d)$. Note that although $Q_r$ is singular, it has an obvious algebraic deformation retraction to $0$, so this na\"ive de Rham complex is still quasi-isomorphic to $R$.

More specifically, we will be working with the twisted de Rham complex 
\[
 \DR(Q_r \by \bA^r)_{\uline{t}.\uline{z}} := (\Omega^*_{Q_r\by \bA^r/R},d + d(\sum_j t_jz_j)\wedge-) 
 \]
 and its pullbacks $\DR(Q_r \by \bA^r \by \overbrace{\OO_r \by \OO_r \by \ldots \by \OO_r}^m)_{\uline{t}.\uline{z}} $. Together these form a cosimplicial cochain complex
 \[
 \DR(Q_r \by \bA^r)_{\uline{t}.\uline{z}} \Rightarrow  \DR(Q_r \by \bA^r \by \OO_r)_{\uline{t}.\uline{z}} \Rrightarrow \DR(Q_r \by \bA^r\by \OO_r\by \OO_r)_{\uline{t}.\uline{z}} \ldots
 \]
whose total complex is $\DR([(Q_r \by \bA^r)/\OO_r])_{\uline{t}.\uline{z}} $. We can also take the twisted de Rham complex with coefficients in the determinant representation $\det$;  the only change is to the cosimplicial structure, and the  the resulting total complex  is $\DR([(Q_r \by \bA^r)/\OO_r], \det)_{\uline{t}.\uline{z}} $.

\begin{lemma}\label{EulerLemmaNew1}
Let $X$ be a dg $R$-manifold equipped with a strictly non-degenerate strict $(-2)$-shifted Poisson structure $\pi$. Then 
 (replacing $X^0$ with an open neighbourhood of $\pi^0X$ if necessary), the global section $\phi \in \Gamma(X^0,\sE)$ from Lemma \ref{StrictPoissonLemmaNew} satisfies $Q(\phi,\phi)=0$ and for the resulting morphism
\[
  X^0 \xra{(\phi,\sE)} [Q_r/\OO_r]
\]
of algebraic stacks as in \S \ref{eulersn}, the right de Rham complex $\DR^r_X((\det \Omega^1_X)^*\ten_{\sO_X}\omega_X)$ on $X^0$ is naturally quasi-isomorphic to the twisted de Rham complex
\[
 \DR(\sO_{X^0})\ten_{(\phi,\sE)^{-1}\DR([Q_r/\OO_r])} (\phi,\sE)^{-1} \DR([(Q_r \by \bA^r)/\OO_r], \det)_{\uline{t}.\uline{z}}[-2\dim X^0].
\]
The  virtual fundamental class
\[
 [e^{\hbar\pi}] \in  \H^{- \dim X}(X,\DR^r_X((\det \Omega^1_X)^*\ten_{\sO_X}\omega_X))\brh \cong \H^r(X^0,X^0\setminus\pi^0X; \DR( \det \sE))\brh
\]
from Corollary \ref{VFCCor} is then given by
\[
 \hbar^{\dim X^0}(\phi,\sE)^* \tilde{e}_r,  
\]
for canonical classes $\tilde{e}_r$ lying in the twisted de Rham cohomology groups
\[
 \tilde{e}_r \in \H^r\DR([(Q_r \by \bA^r)/\OO_r], \det)_{\uline{t}.\uline{z}} ((\hbar)).
\]
\end{lemma}
\begin{proof}
 By Lemma \ref{StrictPoissonLemmaNew}, we know that   $dQ(\phi,\phi)=0$, so $Q(\phi,\phi)$ is locally constant. Since $\phi$ vanishes on  $\pi^0X$, it follows that $Q(\phi,\phi)=0$ on the connected components of $X^0$ containing $\pi^0X$. If we discard the other components, the pair $(\sE,\phi)$ thus defines a   map $X^0 \to [Q_r/\OO_r]$. 

Lemma \ref{QIMconnlemma} and Corollary \ref{DRrtoBMcor} give us quasi-isomorphisms
\[
\DR^r_Y((\det \sE) \ten_{\sO_{X^0}}  \omega_Y)[-\dim X] \xra{\mu(-,\pi)}  \DR^r_{X}(\sO_X,\nabla) \to \DR^r_X(\det \Omega^1_X)^*\ten_{\sO_X} \omega_X)[-\dim X ],
\]
with $\det \Omega^1_X \cong \det \sE^*\ten_{\sO_{X^0}}\sO_X$. Locally on $X^0$, consider the contractible groupoid of left connections on $\sE$, whose objects are connections with a unique isomorphism  between any pair of connections. These give us homotopy inverses $\Psi_{\nabla_{\sE}}$ as in Lemma \ref{QIMconnlemma}.

For  the $\OO_r$-torsor  $\beta \co P \to X^0$ of trivialisations of $\sE$, we have a canonical basis $\{e_j\}_j$ for the orthogonal bundle $\beta^*\sE$, and it is easy to check that we have an isomorphism
\[
 \det \ten \DR(Q_r \by \bA^r)_{\uline{t}.\uline{z}}\ten_{ \phi^{-1}\DR(Q_r)}\DR(P)[-2\dim X^0] \cong  \DR^r_Y((\det \sE) \ten_{\sO_{X^0}}  \omega_Y)\ten_{\DR(\sO_{X^0})}\DR(P) ,
\]
sending $t_j \mapsto \pd_{e_j}$ and $ dt_j\mapsto e_j^*$, noting that $\det$ is generated by $e_1\wedge \ldots \wedge e_r$.   Repeating this construction for $P\by_{X^0}P\cong \OO_r \by P$ and the higher fibre products gives the required quasi-isomorphisms
\begin{align*}
 &\DR(\sO_{X^0})\ten_{(\phi,\sE)^{-1}\DR([Q_r/\OO_r])} (\phi,\sE)^{-1} \DR([(Q_r \by \bA^r)/\OO_r])_{\uline{t}.\uline{z}} \ten\det[-2\dim X^0]\\
 &\cong \DR^r_Y((\det \sE) \ten_{\sO_{X^0}}  \omega_Y\ten_{\DR(\sO_{X^0})}\DR(\cosk_0(P/X^0))\\
 &\simeq  \DR^r_Y((\det \sE) \ten_{\sO_{X^0}}  \omega_Y),
\end{align*}
where $\cosk_0(P/X^0) $ is the $0$-coskeleton of $P$ over $X^0$, i.e. the canonical simplicial scheme resolving $X^0$ by iterated fibre products of $P$.

Now, if we choose a local trivialisation of the vector bundle $\sE$ on $X^0$, with orthonormal basis $\{e_j\}_j$, then locally we can take the  orthogonal  left connection $\nabla_0$  on $\sE$  for which the elements $e_j$ are horizontal.
If we choose local co-ordinates $\{x_i\}_{i=1}^n$ on $X^0$,  the connection $\nabla_0$ then determines co-ordinates $\xi_i \in \sO_{X,2}$ with $\nabla_{0}(e)= \sum_i\{\xi_i,e\}dx_i$ for $e \in \sE=\sO_{X,1}$.  In these co-ordinates, the Poisson structure is given by 
\[
 \pi = \sum_{i=1}^n \pd_{x_i}\pd_{\xi_i} + \half\sum_{j=1}^r \pd_{e_j}\pd_{e_j},
 \]
so as in the proof of Proposition \ref{kappaprop}, we have
\begin{align*}
e^{\hbar\pi}  &=  \exp(\hbar^{-1}\sum_i dx_i d\xi_i)\exp( \half \hbar \sum_{j}\pd_{e_j}^2) \lrcorner  \hbar^nI,\\
 \Psi_{0}(e^{\hbar \pi})&= \hbar^n\exp(\hbar^{-1}\nabla_0(\phi))\exp( \half \hbar \sum_{j}\pd_{e_j}^2).
\end{align*}

Under the maps above, this comes from the cocycle $\hbar^n\tilde{e}_r^0$, where
\[
\tilde{e}_r^0:= \exp( \half \hbar \sum_{j} t_j^2)\prod_{j=1}^r e_j(dt_j +\hbar^{-1} dz_j ) 
 \]
in $\z^r\DR(Q_r \by \bA^r, \det)_{\uline{t}.\uline{z}}((\hbar)) $. That this is indeed closed follows because $\sum_j z_j^2=0$, so  $\sum_j z_j dz_j=0$. 

This cocycle is not $\OO_r$-equivariant, so we need to compare its pullbacks $\pd^0\tilde{e}_r^0, \pd^1 \tilde{e}_r^0$ to $\DR(Q_r \by \bA^r \by \OO_r, \det)_{\uline{t}.\uline{z}} )$ under the two natural maps $\pd_0, \pd_1\co Q_r \by \bA^r \by \OO_r \to Q_r \by \bA^r$
 given by projection and the group action. We then have $\pd^0\tilde{e}_r^0 =\tilde{e}_r^0$,  while $\pd^1\tilde{e}_r^0$  corresponds to a different choice of orthonormal basis $e_k':=\sum_je_jg_{jk}$, or equivalently to the orthogonal left connection $\nabla_0 -g^{-1}.dg$. The canonical homotopy between $\Psi_0$ and $\Psi_{-g^{-1}dg}$ (sending $(d+\delta)\xi_i$ to $g^{-1}\frac{\pd g}{\pd x_i}$) can similarly be rewritten as the pullback along $P\by_{X^0}P \to  Q_r \by \bA^r\by \OO_r$ of a homotopy 
 \[
  \tilde{e}_r^1 \in \DR(Q_r \by \bA^r \by \OO_r, \det)_{\uline{t}.\uline{z}} )^{r-1}
 \]
between $ \pd^0\tilde{e}_r^0 $ and $ \pd^0\tilde{e}_r^1$. We do not need a formula; it suffices to know that it exists canonically as an expression in terms of the co-ordinates on $Q_r \by \bA^r\by \OO_r$.

Uniqueness of the homotopy between choices of connection then guarantees that this will satisfy the cocycle condition, i.e. that
\[
 \pd^1 \tilde{e}_r^1  =\pd^1 \tilde{e}_r^0 + \pd^1 \tilde{e}_r^2 \in \DR(Q_r \by \bA^r \by \OO_r^2, \det)_{\uline{t}.\uline{z}} )^{r-1}((\hbar)).
\]
Thus $\tilde{e}_r:= \tilde{e}_r^0+\tilde{e}_r^1\in \z^r\DR([(Q_r \by \bA^r)/\OO_r], \det)_{\uline{t}.\uline{z}}((\hbar))$ is our required class.

\end{proof}

\begin{lemma}\label{twistedDRlemma}
 There is a canonical $\OO_r$-equivariant quasi-isomorphism between the twisted de Rham complex $ \DR(Q_r \by \bA^r)_{\uline{t}.\uline{z}}$ and the complex 
 \[
  \cocone(\DR(Q_r) \to \oR j_*\DR(Q_r^*)) 
 \]
of sheaves on $Q_r$, where $j \co Q_r^* \to Q_r$ denotes the natural open immersion. In particular, this gives rise to an isomorphism
\[
 \H^*\DR([(Q_r \by \bA^r)/\OO_r], \det)_{\uline{t}.\uline{z}} \simeq \H^*([Q_r/\OO_r],[Q_r^*/\OO_r]; \DR(\sO_{[Q_r/\OO_r]}\ten \det))
\]
to relative cohomology; when $R=\Cx$, the last group is just $\H^*(B\OO_r,B\OO_{r-2}; \det_{\Cx})$.
 \end{lemma}
\begin{proof}
 If we start from the complex $\DR(\sO_{Q_r}[t_1, \ldots,t_r])_{\uline{t}.\uline{z}} )$ of sheaves on $Q_r$, then taking the quotient by the dg subcomplex generated by all multiples of the co-ordinates $t_j$ gives us a quasi-isomorphic quotient complex isomorphic to 
 \[
  \DR(\sO_{Q_r}) \xra{\wedge \sum_j dt_j} \bigoplus_j \DR( \sO_{Q_r}[z_j^{-1}])dt_j \xra{\wedge \sum_j dt_j }  \bigoplus_{j<k} \DR( \sO_{Q_r}[z_j^{-1},z_k^{-1}])dt_j\wedge dt_k \ldots ,
 \]
so the higher terms are effectively the \v Cech complex calculating $\oR j_*\DR(\sO_{Q_r^*})$, and we have a quasi-isomorphism $ \DR(\sO_{Q_r}[t_1, \ldots,t_r])_{\uline{t}.\uline{z}} ) \to \cocone(\DR(Q_r) \to \oR j_*\DR(Q_r^*)) $

Although this description is not $\OO_r$-equivariant, it implies that the complex $ \DR(\sO_{Q_r^*}[t_1, \ldots,t_r])_{\uline{t}.\uline{z}} )$ of sheaves on $Q_r^*$ is acyclic, so the target of the  canonical map 
\[
\DR(\sO_{Q_r}[t_1, \ldots,t_r])_{\uline{t}.\uline{z}} )  \to  \oR j_*\DR(\sO_{Q_r^*}[t_1, \ldots,t_r])_{\uline{t}.\uline{z}} )\by_{\oR j_*\DR(\sO_{Q_r^*})}\DR(\sO_{Q_r})
 \]
 is a model for  $\cocone(\DR(Q_r) \to \oR j_*\DR(Q_r^*))$.
 
 The remaining statements follow by $\OO_r$-equivariance and Lemma \ref{stiefellemma}.
\end{proof}

\begin{lemma}\label{EulerLemmaNew2}
 Under the isomorphisms of Lemma \ref{twistedDRlemma}, when $R=\Cx$, the twisted de Rham cohomology classes
 \[
 \tilde{e}_r \in \H^r\DR([(Q_r \by \bA^r)/\OO_r], \det)_{\uline{t}.\uline{z}} ((\hbar)).
\]
of Lemma \ref{EulerLemmaNew1} correspond  to the Euler classes
\[
 \hbar^{-r/2} e_r \in \H^r(B\OO_r,B\OO_{r-2};{\det}_{\Cx})((\hbar)) 
\]
of  Corollary \ref{stiefelcor} when $r$ is even,  and zero when $r$ is odd.
\end{lemma}
\begin{proof}
 By Corollary \ref{stiefelcor}, we know that for $r$ odd we have $\H^r(B\OO_r,B\OO_{r-2};\det)=0$, while for $r$ even  the space $\H^r(B\OO_r,B\OO_{r-2};\det_{\Cx}$) is spanned by the Euler class $e_r$. Thus $\tilde{e}_r$ is necessarily 
 zero when $r$ is odd, and some multiple of $e_r$ when $r$ is even.
 
Using the isomorphism $ \H^r(B\OO_r,B\OO_{r-2};\det_{\Cx}) \to \H^r(B\OO_r,\det_{\Cx})$ from Corollary \ref{stiefelcor}, it follows that $\tilde{e}_r$ is determined by its image $\bar{e}_r$ in 
 \[
 \H^r\DR([Q_r/\OO_r], {\det}_{\Cx})((\hbar)) \cong \H^r\DR(B\OO_r, {\det}_{\Cx})((\hbar)),
 \]
and indeed by its image in $\H^r\DR(B\SO_r, \Cx)((\hbar)) $, since $\H^r\DR(B\OO_r, \det_{\Cx})((\hbar)) $ is a direct summand as in the proof of Lemma \ref{eulerlemma}.

On pulling back to any scheme $Z$, the class $\bar{e}_r(\sE):= [\sE]^*\bar{e}_r$ can be expressed  in terms of the curvature $\kappa$ of a connection as  $\hbar^{-r}[Q(\kappa)^{r/2}]/(r/2)!$, by Proposition \ref{kappaprop}. The same reasoning holds for smooth connections, giving an expression for $\bar{e}_r(\sE)$ valid for all complex orthogonal vector bundles, including the universal bundles on $B\OO_r$ and $B\SO_r$. 
Since curvature is additive  and exponentials send sums to products, the Whitney sum $\oplus \co (\SO_2)^{r/2} \to \SO_r$  then gives
\[
 {\oplus}^* \bar{e}_r= \bar{e}_2^{\ten (r/2)} \in \H^*(B\SO_2(\Cx), \Cx)^{\ten (r/2)}((\hbar)).
\]

Now,  with respect to the universal real orthogonal $\C^{\infty}$ bundle  $\sU$ on $B\SO_r(\R)$, \cite[Theorem 1.5]{brownCohoBSOnBOn} gives cohomology as  $\H^*(B\SO_{r}(\R),\Q) \cong \Q[p_1, \ldots, p_{(r/2)-1}, e_{r}]$ for $p_i=(-1)^ic_{2i}(\sU \ten \Cx) \in \H^{4i}(B\SO_{r}(\R),\Q)$ (Chern classes) and $e_{r} \in \H^{r}(B\SO_{2s}(\R),\Q)$ (the Euler class), with $(e_{r})^2=p_{r/2}$.
From the fundamental theorem of symmetric polynomials, we can deduce that the Whitney sum gives an injective map 
\[
 \oplus^* \co  \H^*(B\SO_{r}(\R),\Q) \to \H^*((B\SO_2)^{r/2}, \Q)^{S_{r/2}},
\]
with $e_r$ mapping to $e_2^{\ten (r/2)}$ and $p_i$ mapping to the $i$th elementary symmetric function in the variables $1^{\ten j-1}\ten p_1 \ten 1^{(r/2)-j} = 1^{\ten j-1}\ten (e_2)^2 \ten 1^{(r/2)-j}$. 

Meanwhile, the class $\bar{e}_2$ is easy to describe. We have an isomorphism  $\SO_{2,\Cx} \cong \bG_{m, \Cx}$ given by $\left( \begin{smallmatrix}  \alpha  & \beta \\ -\beta & \alpha \end{smallmatrix}\right) \mapsto \alpha + i \beta$. Thus any special orthogonal algebraic  bundle $\sE$  of rank $2$ takes the form $\sE=\sL \oplus \sL^*$. 
In this case, 
the curvature is given by $\kappa(\sL \oplus \sL^*)= \kappa(\sL)\oplus \kappa(\sL^*)$. Under the isomorphism $u \co \det \sE \cong \sO_Z$, we  then have
$uQ(\kappa(\sL \oplus \sL^* ))= \kappa(\sL) = c_1(\sL)$, the first Chern class, so $\bar{e}_2(\sE)= \hbar^{-1}c_1(\sL)$.

Since $S^1 \subset \bG_m(\Cx)$ is a deformation retract, the sheaf $\C^{\infty}(\sL)$ of smooth sections of $\sL$ admits a unitary inner product, so 
\[
\C^{\infty}(\sE) \cong \C^{\infty}(\sL) \oplus \overline{\C^{\infty}(\sL)}= \C^{\infty}(\sL)\ten_{\R}\Cx.
\]
Thus  $\C^{\infty}(\sL)$ is a real form for the complex bundle $\C^{\infty}(\sE) $, and the associated class  $e_2(\sE)$ is the (real) Euler class $e( \C^{\infty}(\sL))$. However, since $\sL$ itself has a complex structure, this Euler class  is just the Chern class $c_1(\sL)$. 

We have therefore shown that $ \bar{e}_2(\sE)= \hbar^{-1}e_2(\sE)$ for all rank $2$ orthogonal bundles $\sE$, and hence that $\bar{e}_2= \hbar^{-1}e_2 $.
Combined with the  Whitney sum formula above, this shows   that $\bar{e}_r=\hbar^{-r/2}e_r$. Since $\tilde{e}_r$ is determined by $\bar{e}_r$, we must also have  $\tilde{e}_r=\hbar^{-r/2}e_r$, as required.
\end{proof}

\subsubsection{Virtual fundamental classes}

Substituting Lemmas \ref{EulerLemmaNew1} and  \ref{EulerLemmaNew2} into Corollary \ref{VFCCor} now gives:
\begin{proposition}\label{EulerPropNew}
Let $X$ be a dg $\Cx$-manifold equipped with a strictly non-degenerate strict $(-2)$-shifted Poisson structure $\pi$, determining an orthogonal  vector bundle $\sE:=\sO_{X,1}$ on $X^0$ and a global section $\phi \in \Gamma(X^0,\sE)$ as in Lemma \ref{StrictPoissonLemmaNew}. Then for quantisations $S$ of $\pi$, 
 the  virtual fundamental classes $[e^S]$
 from Corollary \ref{VFCCor} and Remark \ref{VFClocsysrmk} lie in 
 \[
  \H^{BM}_{\dim X}(\pi^0X(\Cx), o_X\brh),
 \]
Borel--Moore homology
 with coefficients in the local system  
 \[
 o_X:= \{m \in \det \sE^{\an}~:~ Q(m,m) \in \sO_{X^0}^{\an} \text{ locally constant}\}.
 \]
 The virtual fundamental classes $[e^S]$ are given by
\[
 \begin{cases}
  \hbar^{(\dim X)/2}((\phi,\sE)^*e_r)\frown[X^0] (1+\hbar^2\Cx\brh) & 2 \mid \dim X \\
  0 & 2 \nmid \dim X,
 \end{cases}
\]
 for  the Euler classes $e_r \in \H^r([Q_r/\OO_r],[Q_r^*/\OO_r] ;\det_{\Cx}) \cong \H^r(B\OO_r,B\OO_{r-2} ;\det_{\Cx})$  of Corollary \ref{stiefelcor}.
\end{proposition}

\begin{remark}\label{sVrmk}
There is  a more topological characterisation of the map $(\phi,\sE)^*
\co \H^*(B\OO_r,B\OO_{r-2}; \det) \to
\H^*(X^0(\Cx), (X^0 \setminus \pi^0X)(\Cx); o_X)$. First,  choose a real orthogonal $\C^{\infty}$-bundle $\sV$ on the analytic site of $X^0(\Cx)$ with an isomorphism between $\sV\ten_{\R}\Cx$ and the smooth sections of $\sE$; this is possible because $\OO_r(\R) \to \OO_r(\Cx)$ is a deformation retract. Then the real and imaginary parts of $\phi$ give global sections $\mathrm{Re}\, \phi, \,\mathrm{Im}\, \phi$ of $\sV$, and the   equation  $Q(\phi,\phi)=0$ amounts to saying that $\mathrm{Re}\, \phi, \,\mathrm{Im}\, \phi $  are orthogonal and have the same norm. In particular, on $X^0\setminus \pi^0X$ they span an orthogonal subbundle of rank $2$, and we can let $\sW$ be its orthogonal complement, of rank $r-2$.

This gives a commutative diagram 
\[
 \begin{CD}
  X^0 \setminus \pi^0X @>{[\sW]}>> B\OO_{r-2}(\R)\\
@VVV  @VV{ [M] \mapsto [M \oplus \R^2]}V \\
X^0   @>{[\sV]}>>   B\OO_r(\R).
\end{CD}
\]
of topological stacks, and  $(\phi,\sE)^*$ is then just the induced  map on the relative cohomologies.
\end{remark}

\begin{remark}
From the description of cohomology in Corollary \ref{stiefelcor}, note that in the setting of Proposition \ref{EulerPropNew} we must also have a class $(\phi,\sE)^*e_{r
-1} \in \H^{BM}_{\dim X+1}(\pi^0X(\Cx), o_X)$; this does not seem to have an obvious interpretation in terms of quantisations, or indeed an analogue in less strict settings. 
\end{remark}

\begin{corollary}\label{BJcor}
In the setting of Proposition \ref{EulerPropNew}, when $\pi^0X$ is proper and $\sE$  a special orthogonal bundle,  the images in Steenrod homology of the  classes  $[\exp(S)] \in  \H_{\dim X}^{BM}(\pi^0X(\Cx),\Cx)\brh$ associated to   quantisations $S$  of $\pi_{\hbar}$ via  Corollary \ref{VFCCor} are given by 
\[
\hbar^{(\dim X)/2} [X]_{BJ} \cdot (1+\hbar^2 \Cx\brh),
\]
where $[X]_{BJ}$ is the Borisov--Joyce virtual fundamental class $[X_{\mathrm{dm}}]_{\mathrm{virt}}$ of \cite[Corollary 3.19]{BorisovJoyce}.
\end{corollary}
\begin{proof}
Choose a real $\C^{\infty}$ vector bundle $\sV$ as in  Remark \ref{sVrmk}, noting that it must be a special orthogonal bundle since $\det \sE \cong \sO_{X^0}$ by hypothesis.  Observe that the dual vector bundles to $\sV$ and $\sE$ satisfy the conditions of \cite[Definition 3.6]{BorisovJoyce} (in terms of their notation, $\sE$ is given by algebraic sections of $E^*$, and $\sV$ is given by smooth sections of $(E^+)^*$); this relies on the observation that $\mathrm{Re}\, Q$ is positive definite on both $\sV$ and its $\mathrm{Re}\, Q$-orthogonal complement $i \sV$. 
Thus $X_{\mathrm{dm}} := (X^0(\Cx), \sV, \phi) $ determines a Kuranishi neighbourhood of the form in \cite[3.16]{BorisovJoyce}, and $[X]_{BJ}$ is the associated class in Steenrod homology, or equivalently in ordinary homology as $\pi^0X(\Cx)$ is a Euclidean neighbourhood retract (cf.  \cite[Corollary 3.19]{BorisovJoyce}).


In the equivalence \cite[Theorem 4.42]{joyceDmanIntro} 
between bordism and derived bordism classes, $X_{\mathrm{dm}}$ corresponds to the class  of the vanishing locus $Z$ of a generic section of $\sV$, and then $[X]_{BJ}=[Z]$ in Steenrod homology $\H_{\dim X}^{\mathrm{St}}(\pi^0X(\Cx),\Z)= \Lim_i \H_{\dim X}^{\mathrm{St}}(U_i,\Z)$, for a system of open neighbourhoods $U_i$ with  $\pi^0X(\Cx) = \bigcap_i U_i$. Now, by \cite[Proposition 12.8]{BottTu}, the class $[Z]$ is Poincar\'e dual to the Euler class, so  $[Z\cap U_i]= e(\sV) \frown [U_i] \in \H_{\dim X}^{BM}(U_i, \Z)$. The result now follows from Proposition \ref{EulerPropNew} by taking limits.
\end{proof}

\begin{remark}\label{BJrmk}
The construction of Remark \ref{sVrmk} can be adapted to upgrade the  classes in Steenrod homology from \cite[\S 2.6.5]{BorisovJoyce} to classes in Borel--Moore homology, independently of our quantisation results in \S\S \ref{affinesn},\ref{globalsn}. An oriented derived manifold $Y$ with cotangent complex generated in chain degrees $[0,1]$ can be given as the derived vanishing locus of a section $v$ of a   vector bundle $\sV$ on a manifold $Y^0$, with the orientation being an isomorphism $\det \sV \cong \det \Omega^1_{Y^0}$. 
In the setting of Corollary \ref{BJcor}, the section $v$ corresponds to $\mathrm{Re}\, \phi$.  

If $\sV$ has rank $r$, we then have a map $(\sV, \<v\>^{\perp}) \co (Y^0, Y^0 \setminus \pi^0Y) \to (B\GL_r(\R), B\GL_{r-1}(\R))$. There is a  Thom class $e_r \in \H^r(B\GL_r, B\GL_{r-1};\sigma)$, where $\sigma$ is the rank $1$ representation given by the sign of the determinant, and this  pulls back to give a class 
\[
 (\sV, \<v\>^{\perp})^*e_r \in \H^r(Y^0, Y^0 \setminus \pi^0Y; o_{Y^0}) \cong \H^{BM}_{\dim Y^0 - r}(\pi^0Y)  = \H^{BM}_{\dim Y}(\pi^0Y),
\]
where $o_{Y^0}$ denotes the orientation sheaf. 
If we drop the orientation condition on $Y$, then the class instead lies in $\H^{BM}_{\dim Y}(\pi^0Y,o_Y) $, where $o_Y$ is the tensor product of the orientation sheaves of $Y^0$ and of $\sV$.

We expect that Proposition \ref{EulerPropNew} then generalises to arbitrary $(-2)$-shifted symplectic derived Deligne--Mumford stacks, 
so that the classes $[e^S]$ arising from quantisations via Remarks \ref{DRrtoBMRmk} and \ref{QIMconnRmk} should be given in general by  $\hbar^{(\dim Y)/2}(\sV, \<v\>^{\perp})^*e_r(1+\hbar^2\Cx\brh)$ for $r$ even, and by $0$ for $r$ odd. 
\end{remark}

\bibliographystyle{alphanum}
\bibliography{references.bib}

\newcommand{\etalchar}[1]{$^{#1}$}
\def\cprime{$'$}
\begin{thebibliography}{BBBBJ}

\bibitem[BBBBJ]{BBBJdarboux}
O.~Ben-Bassat, C.~Brav, V.~Bussi, and D.~Joyce.
\newblock A ``{D}arboux theorem'' for shifted symplectic structures on derived
  {A}rtin stacks, with applications.
\newblock {\em Geom. Topol.}, 19:1287--1359, 2015.
\newblock arXiv:1312.0090 [math.AG].

\bibitem[BBD{\etalchar{+}}]{BBDJS}
C.~Brav, V.~Bussi, D.~Dupont, D.~Joyce, and B.~Szendr{\"o}i.
\newblock Symmetries and stabilization for sheaves of vanishing cycles.
\newblock {\em J. Singul.}, 11:85--151, 2015.
\newblock arXiv:1211.3259 [math.AG]. With an appendix by J{\"o}rg
  Sch{\"u}rmann.

\bibitem[BG]{BouazizGrojnowski}
E.~Bouaziz and I.~Grojnowski.
\newblock A {$d$}-shifted {D}arboux theorem.
\newblock arXiv:1309.2197v1 [math.AG], 2013.

\bibitem[BJ]{BorisovJoyce}
D.~{Borisov} and D.~{Joyce}.
\newblock {Virtual fundamental classes for moduli spaces of sheaves on
  Calabi-Yau four-folds}.
\newblock {\em Geom. Topol.}, to appear. 2015.
\newblock arXiv: 1504.00690 [math.AG].

\bibitem[BL]{BraunLazarevHtpyBV}
C.~Braun and A.~Lazarev.
\newblock Homotopy {BV} algebras in {P}oisson geometry.
\newblock {\em Trans. Moscow Math. Soc.}, pages 217--227, 2013.

\bibitem[Bro]{brownCohoBSOnBOn}
Edgar~H. Brown, Jr.
\newblock The cohomology of {$B{\rm SO}_{n}$} and {$B{\rm O}_{n}$} with integer
  coefficients.
\newblock {\em Proc. Amer. Math. Soc.}, 85(2):283--288, 1982.

\bibitem[BT]{BottTu}
Raoul Bott and Loring~W. Tu.
\newblock {\em Differential forms in algebraic topology}, volume~82 of {\em
  Graduate Texts in Mathematics}.
\newblock Springer-Verlag, New York-Berlin, 1982.

\bibitem[CFK]{Quot}
Ionu{\c{t}} Ciocan-Fontanine and Mikhail Kapranov.
\newblock Derived {Q}uot schemes.
\newblock {\em Ann. Sci. {\'E}cole Norm. Sup. (4)}, 34(3):403--440, 2001.

\bibitem[CPT{\etalchar{+}}]{CPTVV}
D.~Calaque, T.~Pantev, B.~To{\"e}n, M.~Vaqui{\'e}, and G.~Vezzosi.
\newblock Shifted {P}oisson structures and deformation quantization.
\newblock {\em J. Topol.}, 10(2):483--584, 2017.
\newblock arXiv:1506.03699v4 [math.AG].

\bibitem[FT]{FeiginTsygan}
B.L. Feigin and B.L. Tsygan.
\newblock Additive {K}-theory and crystalline cohomology.
\newblock {\em Functional Analysis and Its Applications}, 19(2):124--132, 1985.

\bibitem[Gai]{GaitsgoryIndCoh}
Dennis Gaitsgory.
\newblock ind-coherent sheaves.
\newblock {\em Mosc. Math. J.}, 13(3):399--528, 553, 2013.
\newblock arXiv: arXiv:1105.4857 [math.AG].

\bibitem[GR]{GaitsgoryRozenblyumCrystal}
Dennis Gaitsgory and Nick Rozenblyum.
\newblock Notes on geometric {L}anglands: crystals and {$D$}-modules.
\newblock arXiv:1111.2087, 2011.

\bibitem[Har]{HartshorneAlgDeRham}
Robin Hartshorne.
\newblock Algebraic de {R}ham cohomology.
\newblock {\em Manuscripta Math.}, 7:125--140, 1972.

\bibitem[Hin]{hinstack}
Vladimir Hinich.
\newblock D{G} coalgebras as formal stacks.
\newblock {\em J. Pure Appl. Algebra}, 162(2-3):209--250, 2001.

\bibitem[Joy]{joyceDmanIntro}
Dominic Joyce.
\newblock An introduction to d-manifolds and derived differential geometry.
\newblock In {\em Moduli spaces}, volume 411 of {\em London Math. Soc. Lecture
  Note Ser.}, pages 230--281. Cambridge Univ. Press, Cambridge, 2014.

\bibitem[Kra]{kravchenko}
Olga Kravchenko.
\newblock Deformations of {B}atalin--{V}ilkovisky algebras.
\newblock In {\em Poisson geometry (Warsaw, 1998)}, volume~51 of {\em Banach
  Center Publ.}, pages 131--139, Warsaw, 2000. Polish Acad. Sci.

\bibitem[Pri1]{DQpoisson}
J.~P. Pridham.
\newblock Quantisation of derived {P}oisson structures.
\newblock arXiv: 1708.00496v2 [math.AG], 2017.

\bibitem[Pri2]{poisson}
J.~P. Pridham.
\newblock Shifted {P}oisson and symplectic structures on derived {$N$}-stacks.
\newblock {\em J. Topol.}, 10(1):178--210, 2017.
\newblock arXiv:1504.01940v5 [math.AG].

\bibitem[Pri3]{DQnonneg}
J.~P. Pridham.
\newblock Deformation quantisation for unshifted symplectic structures on
  derived {A}rtin stacks.
\newblock {\em Selecta Math. (N.S.)}, 24(4):3027--3059, 2018.
\newblock arXiv: 1604.04458v4 [math.AG].

\bibitem[Pri4]{DQvanish}
J.~P. Pridham.
\newblock Deformation quantisation for {$(-1)$}-shifted symplectic structures
  and vanishing cycles.
\newblock {\em Algebr. Geom.}, 6(6):747--779, 2019.
\newblock arXiv:1508.07936v5 [math.AG].

\bibitem[PTVV]{PTVV}
T.~Pantev, B.~To{\"e}n, M.~Vaqui{\'e}, and G.~Vezzosi.
\newblock Shifted symplectic structures.
\newblock {\em Publ. Math. Inst. Hautes \'Etudes Sci.}, 117:271--328, 2013.
\newblock arXiv: 1111.3209v4 [math.AG].

\bibitem[Sai]{saitoDmodsAnalyticSpaces}
Morihiko Saito.
\newblock {$\mathscr{D}$}-modules on analytic spaces.
\newblock {\em Publ. Res. Inst. Math. Sci.}, 27(2):291--332, 1991.

\bibitem[Sim]{simpsonHtpy}
Carlos Simpson.
\newblock Homotopy over the complex numbers and generalized de {R}ham
  cohomology.
\newblock In {\em Moduli of vector bundles ({S}anda, 1994; {K}yoto, 1994)},
  volume 179 of {\em Lecture Notes in Pure and Appl. Math.}, pages 229--263.
  Dekker, New York, 1996.

\bibitem[Vit]{vitagliano}
Luca Vitagliano.
\newblock Representations of homotopy {L}ie-{R}inehart algebras.
\newblock {\em Math. Proc. Cambridge Philos. Soc.}, 158(1):155--191, 2015.

\end{thebibliography}

\end{document}